\documentclass[a4paper,12pt, reqno]{amsart}
\usepackage{amssymb,amsmath}
\usepackage{cite}
\usepackage{amsmath,amscd,amssymb}
\usepackage{latexsym}
\usepackage[colorlinks,citecolor=blue,pagebackref,hypertexnames=false]{hyperref}

\numberwithin{equation}{section}
\allowdisplaybreaks[2]
\theoremstyle{plain}
\newtheorem{theorem}{Theorem}[section]

\newtheorem{lemma}[theorem]{Lemma}
\newtheorem{corollary}[theorem]{Corollary}

\theoremstyle{definition}
\newtheorem{definition}[theorem]{Definition}

\theoremstyle{remark}
\newtheorem{remark}[theorem]{Remark}

\newtheorem{case[theorem]}{Case}

\def \R{{\mathbb R}}

\def \C{{\mathbb C}}

\def\norm#1.#2.{\lVert#1\rVert_{#2}}

\def\R{\mathbb R}

\def \W{{\mathcal W}}

\title[Orthonormal Strichartz inequalities and their applications]{Orthonormal Strichartz inequalities and their applications on abstract measure spaces
}

\author{Guoxia Feng}
\author{Shyam Swarup Mondal} 
\author{Manli Song}
\thanks{Manli Song is the corresponding author.}
\author{Huoxiong Wu}

\address{\endgraf School of Mathematical Sciences, Xiamen University, Xiamen 361005, China}
\email{gxfeng@mail.nwpu.edu.cn}

\address{ 
	\endgraf Stat-Math unit,
	Indian Statistical Institute  Kolkata, 
	BT Road,  Baranagar, Kolkata  700108, India}
\email{mondalshyam055@gmail.com}

\address{\endgraf Research\&Development Institute of Northwestern Polytechnical University in Shenzhen, Sanhang Science\&Technology Building, No. 45th, Gaoxin South 9th Road, Nanshan District, Shenzhen City, 518063}
\address{\endgraf School of Mathematics and Statistics, Northwestern Polytechnical University, Xi'an, Shaanxi 710129, China}
\email{mlsong@nwpu.edu.cn}
\address{\endgraf School of Mathematical Sciences, Xiamen University, Xiamen 361005, China}
\email{huoxwu@xmu.edu.cn}

\date{\today}
\begin{document}
	
	\maketitle

	\allowdisplaybreaks

	\begin{abstract}
		The main objective of this paper is to extend certain fundamental inequalities from a single function to a family of orthonormal systems. In the first part of the paper, we consider a non-negative, self-adjoint operator $L$ on $L^2(X,\mu)$, where $(X,\mu)$ is a measure space.
		Under the assumption that the kernel $K_{it}(x,y)$ of the Schr\"{o}dinger propagator $e^{itL}$ satisfies a uniform $L^\infty$-decay estimate of the form
		\begin{equation*}
			\sup_{x,y\in X}|K_{it}(x,y)|\lesssim |t|^{-\frac{n}{2}},\,|t|<T_0, \text{ for some }n\geq1,
		\end{equation*}
		where $T_0\in(0,+\infty]$,
		we establish Strichartz estimates for the Schr\"{o}dinger propagator $e^{itL}$ and using a duality principle argument by Frank-Sabin \cite{FS}, we extend it for a system of infinitely many fermions on $L^2(X)$.    We also obtain orthonormal Strichartz estimates for a  class of dispersive semigroup $U(t)=e^{it\phi(L)}\psi(\sqrt{L}),$ where   $\phi: \mathbb{R}^+\rightarrow \mathbb{R}$ is a smooth function and $\psi\in C_c^\infty([\frac{1}{2},2])$. As an application of these orthonormal versions of Strichartz estimates, we prove the well-posedness for the Hartree equation in the Schatten spaces.

		In the next part of the paper, we obtain some new orthonormal Strichartz estimates, which extend prior work of Kenig-Ponce-Vega \cite{Kenig-Ponce-Vega} for single functions.  Using those  orthonormal versions of Kenig-Ponce-Vega result,  we  prove the orthonormal  restriction theorem for the Fourier transform on  some particular noncompact hypersurface of the form   
		$S=\{(\xi, \phi(\xi): \xi\in \R)\}$,     where $\phi$  satisfies certain growth condition. 
	\end{abstract}
	\tableofcontents 	
	
	\section{Introduction}
	\subsection{Orthonormal Strichartz estimates}	For a  measure space   $(X,\mu)$,  the first part of this paper is concerned to derive Strichartz estimates  for a non-negative self-adjoint operator  $L$  on $L^2(X)$ of the  form
	\begin{equation}\label{eq11111}
		\left\|\sum_{j\in J}n_j|e^{itL}f_j|^2\right\|_{L^q(\mathcal{I},L^p(X))} \leq C\|\{n_j\}_{j\in J}\|_{\ell^{\lambda}},
	\end{equation}
	for families of orthonormal data  $\{f_j\}_{j\in J}$ in $L^2(X)$, $\{n_j\}_{j\in J} \in \ell^{\lambda}$  and $p, q, \lambda \in [1, \infty]$ satisfy some particular
	conditions.

	The motivation to investigate generalizations of fundamental estimates to such orthonormal estimates of the form (\ref{eq11111}) came from the theory of many body quantum mechanics.
	A system of $N$ independent fermions in the Euclidean space $\mathbb{R}^{d}$ can be represented in quantum mechanics by a set of $N$ orthonormal functions $u_{1}, \ldots, u_{N}$ in the Hilbert space $L^{2}\left(\mathbb{R}^{d}\right)$. It is consequently critical to develop functional inequalities on these systems with optimal behavior in the finite number $N$ of such orthonormal functions.  For this specific reason, the mathematical study of large quantum systems finds great advantage in functional inequalities involving system of orthonormal functions. Note that the idea of generalizing classical inequalities from a single-function input to an orthonormal family is not a new topic, and the first pioneering work of this kind of generalization can be traced back to the work of  Lieb-Thirring  \cite{lieb, liebb}.  This Lieb-Thirring inequality generalizes extended versions of certain Gagliardo–Nirenberg–Sobolev inequalities, which is considered to be one of the most important tools for the study of the stability of matter \cite{lieb}.  There are several reasons to analyze estimates of the form (\ref{eq11111}). Particularly,  application to the Hartree equation, which models an infinite number of fermions in a quantum system, can be found in  \cite{BLN, LS1, LS2}. 
	
	In recent years, a significant attention has been devoted by numerous researchers to investigate Strichartz estimates in the framework of many body quantum systems.  In this direction, the Strichartz estimates for orthonormal families of initial data was first proved by Frank-Lewin-Lieb-Seiringer \cite{frank}  for the free Schr\"odinger equation  of the form 
	\begin{equation}\label{eq111}
		\left\|\sum_{j\in J} n_j|e^{it\Delta}f_j|^2\right\|_{L^q(\mathbb{R},L^p(\mathbb{R}^d))} \leq C \|\{n_j\}_{j\in J}\|_{\ell^\frac{2p}{p+1}},
	\end{equation}
	where $\{f_j\}$ is an  orthonormal  system $L^2(\mathbb{R}^d)$, $\{n_j\}_{j\in J}\in \ell^\frac{2p}{p+1}$ and $p, q\in [1, \infty]$ satisfy  some particular conditions. The estimate \eqref{eq111}  significantly generalizes the   Strichartz estimate in the classical setting for a single function. This orthonormal estimate \eqref{eq111} was further extended in different contexts, see  \cite{BLN, FS, BHLNS}.  Later, orthonormal Strichartz estimates for some special operators such as the Hermite operator, the twisted Laplacian, the Laguerre operator, the $(k, a)$-generalized Laguerre operator, and the Dunkl operator can be found in \cite{MS}, \cite{GMS}, \cite{GM}, \cite{MM}, and \cite{JMMP}, respectively.  We also refer to \cite{Hos} for orthonormal strichartz estimates for operators with general potential.  Note that the dispersive estimate property for the kernel of the target operator plays a key role in obtaining the orthonormal inequalities.  Motivated by the recent works in the direction of the orthonormal system of initial data, here,   under some suitable dispersive assumptions on the kernel of the general Schr\"odinger propagator, we want to investigate such types of orthonormal estimates in a more general framework, mainly on general measure spaces.

	Assume that $(X,\mu)$ is a measure space. For $1\leq p\leq \infty$, we  denote the space  $L^p(X)$ of all such function $f: X\rightarrow \mathbb{C}$ such that $\|f\|_{L^p(X)}<\infty$, where 
	\begin{equation*}
		\|f\|_{L^p(X)}=
		\begin{cases}\left(\int_X |f(x)|^pd\mu(x)\right)^\frac{1}{p},&1\leq p<\infty\\
			ess \sup\limits_{x\in X} |f(x)|,& p=\infty.
		\end{cases}
	\end{equation*}

	In the study of Strichartz estimates for Schr\"odinger or wave equation,  the kernel of the Schr\"odinger semigroup plays a fundamental role. Taking this into account, in this paper we also make
	the following assumption for the study of the Schr\"odinger equations in the
	abstract setting.
	
	Suppose $L$ is a non-negative self-adjoint operator on $L^2(X)$. Then the Schr\"odinger operator $e^{itL}$ is an unitary operator on $L^2(X)$. We assume that  
	
	\underline{Assumption $(\mathcal{A})$} $e^{itL}$ is  an integral operator with kernel $K_{it}(x,y)$, i.e., 
	\begin{equation*}
		e^{itL}f(x)=\int_X K_{it}(x,y)f(y)d\mu(y), \quad   \text{a.e.}~x\in X,
	\end{equation*}
	for every $f\in L^2(X)$. 
	
	\underline{Assumption $(\mathcal{B})$} The kernel $K_{it}(x,y)$ of the Schr\"odinger operator $e^{itL}$ satisfies an uniform dispersive estimate:
	\begin{equation*}
		\sup_{x,y\in X}|K_{it}(x,y)|\lesssim |t|^{-\frac{n}{2}},\,|t|<T_0, \text{ for some }n\geq1,
	\end{equation*}
	where $T_0\in(0,+\infty]$. This is frequently the case for many important operators, notably the Laplacian $L=-\Delta$ on the Euclidean space $\R^n$ and its potential perturbations, such as the Hermite operator, the twisted Laplacian, and the Laguerre operator, etc. 
	
	Denoting  $\mathcal{I}=(-T_0,T_0)$,   the mixed space-time norm is defined in the usual way by
	\begin{equation*}
		\|f\|_{L^q(\mathcal{I},L^p(X))}=\left(\int_\mathcal{I}\left(\int_X|f(x,t)|^pd\mu(x)\right)^\frac{q}{p}dt\right)^\frac{1}{q},
	\end{equation*}
	with the obvious modifications if $p=\infty$ or $q=\infty$, in which case we use the supremum.
	
	\begin{remark}In the famous work by Keel-Tao \cite{KT}, in order to study Strichartz estimates for a general class of energy conservation operators $U(t)$ $(t\in\mathbb{R})$ on general measure spaces, they assumed that the $L^1\rightarrow L^\infty$ norm of $U(t)$ obeys a dispersive estimate. Additionally, if $X$ is Separable measure space,  by \cite{Simon1982}, the dispersive estimate on the $L^1\rightarrow L^\infty$ norm of $U(t)$ is equivalent to the estimate on its kernel. Since the dispersive type estimate assumption on the kernel of $U(t)$ is an essential tool in establishing orthonormal inequalities, for this reason, we also make the dispersive assumption on its kernel.
	\end{remark}
	
	Then under the above assumptions, our first main result of this paper is the following orthonormal inequality in general measure spaces. 
	\begin{theorem}\label{OSI}
		Assume $L$ is a non-negative self-adjoint operator on $L^2(X)$ satisfying Assumption $(\mathcal{A})$ and $(\mathcal{B})$. Suppose $p,q\geq 1$ satisfies \begin{equation*}
			\frac{2}{q}+\frac{n}{p}=n. 
		\end{equation*}
		(i) If  $  1\leq p<\frac{n+1}{n-1},$  then 
		\begin{equation}\label{eq1}
			\left\|\sum_{j\in J}n_j|e^{itL}f_j|^2\right\|_{L^q(\mathcal{I},L^p(X))} \leq C\left\|\{n_j\}_{j\in J}\right\|_{\ell^\frac{2p}{p+1}},
		\end{equation}
		holds for any (possibly infinite) orthonormal system $\{f_j\}_{j\in J}$ in $L^2(X)$ and any sequence $\{n_j\}_{j\in J}$ in $\mathbb{C}$,
		where $C>0$ is independent of $\{f_j\}_{j\in J}$ and $\{n_j\}_{j\in J}$.\\
		
		(ii) If 
		\begin{align*}
			\begin{cases}  1\leq n<2, \\
				\frac{n+1}{n-1}\leq p\leq \infty,
			\end{cases}or ~~ \begin{cases}  n=2  \\
				3\leq p<\infty,
			\end{cases} or ~~ \begin{cases}  n>2\\
				\frac{n+1}{n-1}\leq p\leq \frac{n}{n-2},
			\end{cases}
		\end{align*}   then for  any $1\leq\lambda<q$
		\begin{equation}\label{eq2}
			\left\|\sum_{j\in J}n_j|e^{itL}f_j|^2\right\|_{L^q(\mathcal{I},L^p(X))} \leq C\left\|\{n_j\}_{j\in J}\right\|_{\ell^\lambda},
		\end{equation}
		holds for any (possibly infinite) orthonormal system $\{f_j\}_{j\in J}$ in $L^2(X)$ and any sequence $\{n_j\}_{j\in J}$ in $\mathbb{C}$,
		where $C>0$ is independent of $\{f_j\}_{j\in J}$ and $\{n_j\}_{j\in J}$.
	\end{theorem}
	Moreover, argued similarly as the proof of the above theorem, using the duality principle in Lorentz spaces and Strichartz estimates for a single function in Lorentz spaces (see Remark \ref{Strichartz-Lorentz}), we can also prove orthonormal Strichartz versions in Lorentz spaces,  which is in fact crucial to obtain the orthonormal Strichartz estimates of initial data with regularity for the classical Schr\"odinger operator in \cite{BHLNS}. Furthermore,  our approach can also be applied to obtain orthonormal Strichartz estimates for a general class of dispersive semigroup associated with $L$. Consider the truncated (or called frequency localized) dispersive semigroup $U(t)=e^{it\phi(L)}\psi(\sqrt{L})$, where $\phi: \mathbb{R}^+\rightarrow \mathbb{R}$ is a smooth function and $\psi\in C_c^\infty([\frac{1}{2},2])$.  Here, we assume $U(t)$ is an integral operator  on $L^2(X)$ with kernel satisfies dispersive decay condition similarly as in  {Assumption $(\mathcal{B})$}. We refer to \cite{Feng-Song-Wu, Song-Yang, Song-Tan} for a detailed dispersive type estimate for the kernel of the semigroup $U(t)=e^{it\phi(L)}\psi(\sqrt{L})$ on abstract measure spaces and H-type groups.   Keeping in mind that kernel of $U(t)$ satisfies dispersive decay condition (similarly  to {Assumption $(\mathcal{B})$}), by Keel-Tao \cite{KT}, we obtain the following Strichartz estimate for $U(t)$ involving a single function 
	\begin{equation}\label{gebneral}
		\|U(t)f\|_{L^{2q}(\mathcal{I},L^{2p}(X))}\lesssim_{\phi,\psi} \|f\|_{L^2(X)},
	\end{equation}
	holds for any $p,q\geq1$ satisfying $\frac{2}{q}+\frac{n}{p}=n$ except for $(q,p,n)=(1,\infty,2)$. Since orthogonality is not preserved by the action of the frequency localized operator, we cannot directly apply Theorem  \ref{OSI} to get the orthonormal version of (\ref{gebneral}), however,  using similar techniques as in the proof of Theorem  \ref{OSI}, we will extend (\ref{gebneral}) for a system of orthonormal functions. Indeed, these orthonormal frequency localized estimates are quite useful to establish the orthonormal Strichartz estimates of initial data with regularity. In particular, when $L=-\Delta$,  the classical Laplacian on $X=\mathbb{R}^n$ and $\phi(r)=r$, Bez-Hong-Lee-Nakamura-Sawano in \cite[Theorem 1.4]{BHLNS} proved the frequency localized estimates for the Schr\"odinger operator $e^{it\Delta}$  and obtained orthonormal Strichartz estimates of initial data with regularity. In this direction, we will investigate orthonormal Strichartz estimates of initial data with regularity for a non-negative self-adjoint operator on abstract measure spaces in our forthcoming paper.

	Notice that $u(x,t)=e^{itL}f(x)$  is the solution to the Schr\"{o}dinger equation given by 
	\begin{equation*}
		\begin{cases}
			i\partial_t u(x,t)+L u(x,t)=0\\
			u(0,t)=f(x).
		\end{cases}
	\end{equation*}
	In the case of a finite number $M$ of particles, such a system is modelled by $M$ orthonormal functions $u_1, \ldots, u_M$ in $L^2(X)$ satisfying  the following system of Hartree equations	 with respect to a potential $w$
	$$
	\left\{\begin{array}{rlrl}
		i \partial_t u_1 & =\left(L+w * \rho\right) u_1, & & \left.u_1\right|_{t=0}=f_1 \\
		& \vdots \\
		i \partial_t u_M & =\left(L+w * \rho\right) u_M, & & \left.u_M\right|_{t=0}=f_M,
	\end{array}\right.
	$$
	where $(x, t) \in  X \times \mathbb{R},\left\{f_j\right\}_{j=1}^M$ is an orthonormal system in $L^2\left(X\right)$ and $\rho$ is a density function defined by $\rho(x, t)=\sum\limits_{j=1}^M\left|u_j(x, t)\right|^2$.  One of the natural and interesting questions is to understand the behavior of the solution to the above system in the case when $M \rightarrow \infty$.  Regarding this, in this paper, we prove the global existence of the solution to the above Hartree equation for infinitely many particles as an application of the inhomogeneous orthonormal Strichartz estimates associated with the non-negative self-adjoint operator $L$.

	\subsection{Orthonormal smoothing estimates}
	Understanding how geometry affects the behavior of solutions to the nonlinear evolution of partial differential equations is a key problem in the field of analysis.   As we discussed earlier, the Strichartz inequality and the smoothing property are the two most stable methods used in the analysis of nonlinear dispersive equations. Strichartz inequality often refers to mixed $L^p-L^q$ space-time norm estimates of the solution, while smoothing properties are Sobolev $L^2$ space-time estimates.  More precisely,   the Strichartz estimate provides the space-time control of solutions while the smoothing estimate allows us to gain derivatives of solutions on any bounded domains.    Typically one employs smoothing estimates to overcome the loss of derivatives caused by the semilinear terms from the nonlinearity.  Nevertheless, both approaches entirely  rely on global analysis for solutions to the corresponding linearized equations.  Exploring different smoothing properties of dispersive-wave equations has been greatly enhanced by the discovery of the local Kato-smoothing property. This has led to an improvement in the mathematical theory of nonlinear dispersive-wave equations, specifically the well-posedness of their Cauchy problems in low regularity spaces.

	Consider the class of dispersive Cauchy problem
	\begin{align}\label{smo}
		\left\{\begin{array}{l}
			\partial_t u-i P(D) u=0, \quad x \in \mathbb{R}^d, t \in \mathbb{R}, \\
			u(x, 0)=u_0(x),
		\end{array}\right.
	\end{align}
	where $D=-i \nabla_x$.  The solution of the Cauchy problem (\ref{smo}) is given by $e^{itP(D)}u_0$, where the  operator $P(D)$ is defined through its Fourier symbol $P(\xi)$, i.e., 
	$$P(D)f(x)=\int_{\mathbb{R}^d} e^{ix\xi}P(\xi)\hat{f}(\xi)d\xi.$$    
	When the  symbols $P(\xi)$   roughly behaves like $|\xi|^m$ as   $|\xi |\to \infty$ with $m>1$, then   the local smooting effect happens  for  the solution $u$ to (\ref{smo}), i.e., if $u_0\in H^s(\mathbb{R}^d)$ then the solution $u$  has $(m-1)/2$ derivative in $L^\infty(\mathbb{R}^d)$.  This local smoothing effect, as noted by Constantin-Saut in \cite{CS}, is due to the dispersive nature of the linear part of the equation, and the gain of the regularity $(m-1)/2$,  is solely determined by the order $m$ of the equation and has nothing to do, in particular, with the dimension of the spatial domain $d$.   In fact, the higher the order $m$, the more dispersive the equation is and the stronger the local smoothing effect becomes.  
	
	Furthermore, if $P(\xi)=(-\Delta)^m+R(\xi)$, where $R$ is a pseudo-differential operator whose symbol  belongs to the class of symbols $S_{1,0}^{2 m-1}$, then it was proved in  \cite{BAL}    that   for $u_0 \in L^1\left(\R^d\right)$, the solution  $u $ has $r(m)$ derivatives in $L^{\infty}\left(\mathbb{R}^d\right)$, where \begin{itemize}
		\item $r(m) \leq(m-1)(d-1)-2~$ if $n$ is odd. 
		\item $r(m) \leq(m-1)(d-2)-2~$ if $n$ is even.
	\end{itemize}  
	For the solution of the  linear KdV equation (Airy equation), i.e., $P(\xi)=\xi^3$ with $n=1$, 
	\begin{itemize}
		\item  For  $u_0 \in L^1(\mathbb{R})$ data, $\frac{1}{2}$ derivatives of $u$ in $L^{\infty}(\mathbb{R})$ a.e. $t$, see \cite{KP}.
		\item  For  $u_0 \in L^2(\mathbb{R})$ data, $\frac{1}{4}$ derivatives  of $u$ in $L^{\infty}(\mathbb{R})$ a.e. $t$, see \cite{KPV} .
	\end{itemize} 
	The smoothing effect for the Korteweg-de Vries equation was originally demonstrated by Kato in \cite{Kato}, but for Schr\"odinger equations, the homogeneous smoothing effect was simultaneously established by Sj\"olin 
	\cite{Sj} and Vega \cite{Vega88}. Constantin and Saut investigated the homogeneous smoothing effect for general dispersive equations in \cite{CS}, while Kenig, Ponce, and Vega improved and generalized it in \cite{Kenig-Ponce-Vega}. For example, Kenig-Ponce-Vega in \cite{Kenig-Ponce-Vega} extend the  above results of \cite{KPV}  concerning the global smoothing effect of solutions  to \eqref{smo} with $d=1$ to general differential operators of the form 
	\begin{equation}\label{higher}
		\mathcal{W}_\nu(t) f(x)=\int_{\Omega} e^{i\left(t\phi(\xi)+x\xi\right)}|\phi''(\xi)|^\frac{\nu}{2} \hat{f}(\xi) d\xi, 
	\end{equation} for $\nu\geq 0$ and  $\phi$ is from  the  class $\mathcal{A}$, see Definition \ref{Definition} for details about the class $\mathcal{A}$  and   for any other unknown   notations used here. In a particular case,    Kenig-Ponce-Vega also proved that if the initial data $u_0 \in L^2(\mathbb{R})$, then $\left|P^{\prime \prime}(D)\right|^{1 / 4} u(\cdot, t) \in L^{\infty}(\mathbb{R})$ a.e. $t$.  The two primary standard approaches to derive the above smoothing estimates are the restriction theorem and resolvent estimate (or its variants). However,   regarding the operator $\mathcal{W}_\nu(t)$, in \cite{Kenig-Ponce-Vega},  the authors   also  proved     
	the following smoothing estimates for a single function in one dimension. 
	\begin{theorem} \label{SS1}
		Let $\phi\in\mathcal{A}$. Then for any $\nu\in [0,1]$ and $(p,q)=\left(\frac{2}{1-\nu},\frac{4}{\nu}\right)$,
		\begin{align}\label{S1}
			\|\mathcal{W}_\frac{\nu}{2}(t) f\|_{L^q(\mathbb{R},L^p(\mathbb{R})}&\leq C\|f\|_{L^2(\mathbb{R})}
		\end{align}
		for a constant $C>0$ independent of $f$. 
	\end{theorem}
	
	\begin{remark}
		Its important to note that, the operator  $\mathcal{W}_\nu(t)$ does not satisfy the semigroup property, i.e., $\mathcal{W}_\nu(t)\mathcal{W}_\nu(s) \neq \mathcal{W}_\nu(t+s)$, while it still obeys the similar type property $\mathcal{W}_\nu(t)\mathcal{W}_\nu(s)=\mathcal{W}_{2\nu}(t+s)$. Even if $\mathcal{W}_\nu(t)$ does not form a semigroup, using the dispersive estimates for the kernel of $\mathcal{W}_\nu(t)$ (see Lemma \ref{decay-1}) along with the complex interpolation and standard duality arguments, one can obtain  \eqref{S1}.
	\end{remark} 
	Indeed, \eqref{S1} generalizes the known classical Strichartz estimates proved by Strichartz in \cite{RS} for one-dimensional Schr\"odinger equations.  Here we also want to emphasize that the operator $\mathcal{W}_{\nu}(t) $  also can be seen as a restriction of the Fourier transform on certain noncompact hypersurface $S$  of $\mathbb{R}^2$. In fact, if we consider   the   hypersurface  as 
	$$S=\{(\xi, \phi(\xi): \xi\in\mathbb{R}\},$$
	endowed with the measure
	$$d\sigma_\nu(\xi,\phi(\xi))=\frac{|\phi''(\xi)|^\frac{\nu}{2}}{\sqrt{1+|\phi'(\xi)|^2}}d\xi,$$
	then for all $F\in L^1(S,d\sigma_\nu)\cap L^2(S,d\sigma_\nu)$ and $(x,t)\in\mathbb{R}^2$,   one   has the following identity 
	\begin{align*}
		E^\nu_SF(x,t)&=\frac{1}{4\pi^2}\int_S e^{i(x,t)\cdot (\xi,\phi(\xi))} F(\xi,\phi(\xi))d\sigma_\nu(\xi,\phi(\xi))\\
		&=\frac{1}{4\pi^2}\int_\mathbb{R}e^{i\left(t\phi(\xi)+x\xi\right)} F(\xi,\phi(\xi))|\phi''(\xi)|^\frac{\nu}{2} d\xi.
	\end{align*} 
	In particular, choosing $f:\mathbb{R}\rightarrow \mathbb{C}$ such that $\hat{f}(\xi)=|\phi''(\xi)|^\frac{\nu}{4}F(\xi,\phi(\xi))$, we get
	\begin{equation}\label{relation}
		E^\nu_SF(x,t)=\frac{1}{4\pi^2}\mathcal{W}_\frac{\nu}{2}(t) f(x)
	\end{equation} and therefore, the Strichartz estimates (\ref{S1}) actually reduces to the 
	following restriction estimate.
	\begin{theorem}\label{R1}
		Suppose $\phi\in\mathcal{A}$, $\nu\in [0,1]$ and $(p,q)=\left(\frac{2}{1-\nu},\frac{4}{\nu}\right)$.
		Let $S=\{(\xi, \phi(\xi): \xi\in\mathbb{R}\}$
		be the noncompact hypersurface of $\mathbb{R}^2$ endowed with the measure $d\sigma_\nu(\xi,\phi(\xi))=\frac{|\phi''(\xi)|^\frac{\nu}{2}}{\sqrt{1+|\phi'(\xi)|^2}}d\xi$. Then for any $F\in L^2(S,d\sigma_\nu)$,
		\begin{equation*}
			\|E^\nu_SF\|_{L^q(\mathbb{R}, L^p(\mathbb{R}))}\lesssim \|F\|_{L^2(S,d\sigma_\nu)}.
		\end{equation*}
	\end{theorem}

	One of the classical problems in harmonic analysis is the so-called restriction problem. The classical restriction problem   is the following: Let $S\subset\R^d, d\geq 2$ be a hypersurface endowed with its $(d-1)$-dimensional Lebesgue measure $d\mu,$ then one may ask the following:   For what exponents $1 \leq p \leq 2$, do we have the inequality
	\begin{equation}\label{Restriction1}
		\int_S|\hat{f}(\xi)|^2 d\mu(\xi)\leq C \|f\|^2_{L^p(\R^d)}?
	\end{equation}
	Here  $\hat{f}$ denotes the Fourier transform for a Schwartz class function $f$ on $\R^d$ and  is defined as
	$$\hat{f}(\xi)=\frac{1}{(2\pi)^{n/2}}\int_{\R^d}f(x)e^{-ix\cdot\xi}dx, \quad \xi \in \mathbb{R}^d.$$

	The restriction phenomena were first introduced by E. M. Stein for studying the boundedness of the Fourier transform of an $L^p$-function in the Euclidean space $\mathbb{R}^d$ for some $d\geq2$. It has a delicate connection to many other conjectures, notably the Kakeya and Bochner-Riesz conjectures. Furthermore, from \cite{RS} we see that it is also closely related to that of estimating solutions to linear PDE such as the wave and Schr\"odinger equations.  
	
	Let   $R_S$ be the operator defined as $R_Sf=\hat{f}|_S$, known as the restriction operator. Then the restriction problem  (\ref{Restriction1}) is equivalent to the boundedness of $R_S: L^p(\R^d) \to L^2(S).$ We also denote the operator $E_S$  (as the dual of the restriction operator $R_S$) is defined as
	$$E_S f(x)=\frac{1}{(2\pi)^{d/2}}\int_S f(\xi)e^{i\xi\cdot x} d\mu(\xi), \quad x\in\R^d,$$ for $f\in L^1(S).$
	If  $q$ is the conjugate exponent of  $p$, i.e., $\frac{1}{p}+\frac{1}{q}=1,$ then by duality, (\ref{Restriction1})   can also be posed  for  $d\geq 2$ as follows:		For what exponent $2\leq q\leq \infty,$ is it true that for all $f\in L^2(S)$, we have  
	\begin{align}\label{11}
		\|E_S f\|_{L^q(\R^d)}\leq C\|f\|_{L^2(S)}?\end{align}

	The case in which the restriction problem is often considered in the literature is $q=2$.  From the literature, it is well known that there are only two types of surfaces for which the restriction problem (\ref{Restriction1}) (or (\ref{11}))  has been settled completely. Namely, for smooth compact surfaces with non-zero Gauss curvature and quadratic surfaces.   For smooth compact surfaces,   Stein-Tomas theorem  \cite{ES, PT} asserts that the restriction conjecture  (\ref{Restriction1}) holds for all $1\leq p\leq \frac{2(d+1)}{d+3}.$  For  quadratic surfaces such as paraboloid-like, cone-like, or sphere-like,  Strichartz gave a complete characterization depending on surfaces in \cite{RS}. We refer to the excellent review by Tao \cite{tao} for a complete study on the restriction problems.

	We can see that (\ref{11}) is the restriction estimate for a single function, and as we discussed earlier, our focus is to generalize it for a system of orthonormal functions.  More precisely, for   $\{f_j\}_{j\in J}$ a (possibly infinite) orthonormal system in  $L^2(S)$ and  $\{n_j\}_{j\in J}\subset \C$,  we want to find estimates
	of the form \begin{equation}\label{rec}
		\left\|\sum_{j\in J}n_j|E_S f_j|^2\right\|_{L^{q'/2}(\R^d)}\leq C \|\{n_j\}_{j\in J}\|_{\ell^\lambda},
	\end{equation}
	for some $1\leq q<2, \lambda>1$ with $C>0$ independent of $\{n_j\}_{j\in J}$ and $\{f_j\}_{j\in J}$ and the value of $\lambda$ depending on the surface $S$.  In this direction, our next main results of this paper are the extension of Strichartz estimates in Theorem \ref{SS1} and restriction estimate in Theorem \ref{R1} for a system of orthonormal functions as follows.  
	\begin{theorem}[Orthonormal Strichartz estimates]\label{real line}
		Let $\phi\in\mathcal{A}$, $\nu\in [\frac{2}{3},1)$ and $(p,q)=\left(\frac{1}{1-\nu},\frac{2}{\nu}\right)$. 
		For any (possibly infinite) orthonormal system $\{f_j\}_{j\in J}$ in $L^2(\mathbb{R})$ and any sequence $\{n_j\}_{j\in J}$ in $\mathbb{C}$, we have
		\begin{equation*}
			\left\|\sum_{j\in J}n_j|\mathcal{W}_\frac{\nu}{2}(t)f_j|^2\right\|_{L^q(\mathcal{\mathbb{R}},L^p(\mathbb{R}))} \leq C\|\{n_j\}_{j\in J}\|_{\ell^\frac{2p}{p+1}},
		\end{equation*}
		where $C>0$ is independent of $\{f_j\}_{j\in J}$ and $\{n_j\}_{j\in J}$.
	\end{theorem}
	
	\begin{theorem}[Orthonormal restriction estimates]\label{orthonormal-restriction}
		Suppose $\phi\in\mathcal{A}$, $\nu\in [0,1]$ and $(p,q)=\left(\frac{2}{1-\nu},\frac{4}{\nu}\right)$.
		Let $S=\{(\xi, \phi(\xi): \xi\in\mathbb{R}\}$
		be the noncompact hypersurface of $\mathbb{R}^2$ endowed with the measure $d\sigma_\nu(\xi,\phi(\xi))=\frac{|\phi''(\xi)|^\frac{\nu}{2}}{\sqrt{1+|\phi'(\xi)|^2}}d\xi$. 
		For any (possibly infinite) orthonormal system $\{F_j\}_{j\in J}$ in $L^2(S, d\sigma_\nu)$ and any sequence $\{n_j\}_{j\in J}$ in $\mathbb{C}$, we have
		\begin{equation*}
			\left\|\sum_{j\in J}n_j|E^\nu_SF_j|^2\right\|_{L^q(\mathcal{\mathbb{R}},L^p(\mathbb{R}))} \leq C\|\{n_j\}_{j\in J}\|_{\ell^\frac{2p}{p+1}},
		\end{equation*}
		where $C>0$ is independent of $\{F_j\}_{j\in J}$ and $\{n_j\}_{j\in J}$.
	\end{theorem}

	Coming back to the global smoothing effect, the authors in  \cite{Kenig-Ponce-Vega} also extended  $ \mathcal{W}_\nu(t) $ defined in  (\ref{higher}) into the higher dimension $\mathbb{R}^d$.  The first problem in the higher dimension was to find the substitute for the second derivative that appeared in the dimension case; however, as explained in \cite{Kenig-Ponce-Vega}, the stationary phase lemma suggests that this substitute should be the Hessian of the phase function. For $\nu\geq 0$, they considered the operator 
	\begin{equation*}
		\mathcal{W}_\nu(t) f(x)=\int_{\mathbb{R}^d} e^{i\left(t P(\xi)+x\cdot\xi\right)}|HP(\xi)|^\frac{\nu}{2} \hat{f}(\xi) \psi_P(\xi)d\xi,
	\end{equation*}
	for a    suitable cut-off function $\psi_P$ and proved the Strichartz type estimates
	\begin{align}\label{S2}
		\|\mathcal{W}_\frac{\nu}{2}(t) f\|_{L^q(\mathbb{R},L^p(\mathbb{R}^d))}&\leq C\|f\|_{L^2(\mathbb{R}^d)}. 
	\end{align}
	for   $0\leq \nu<\frac{2}{d}$ and $(p,q)=\left(\frac{2}{1-\nu},\frac{4}{d\nu}\right)$ and under some assumption on the polynomial $P$, where     $HP=\det(\partial_{ij}^2P)_{i,j}$ denotes   the determinant of	the Hessian matrix of $P$.  	   In fact, if $P$ is a real polynomial, it is not easy to see what the Hessian of $P$ looks like, but,  things will be simple if we consider    $P$  to be a a real elliptic polynomial.    Note that 
	\begin{itemize}
		\item if $u_0 \in L^1\left(\mathbb{R}^d\right)$,      $\frac{n}{2}(m-2)$ derivatives of    $u$ in $L^{\infty}\left(\mathbb{R}^n\right)$,  
		\item  if $u_0 \in L^2\left(\mathbb{R}^d\right)$,    $\frac{d}{2}(m-2)\left(\frac{1}{2}-\frac{1}{p}\right)$ derivatives  of $u$ in  $L^p\left({\R}^d\right)$ with $0 \leq \frac{1}{2}-\frac{1}{p}<\frac{2}{d}$.
	\end{itemize} 
	Our next aim in this paper is to extend the estimates  (\ref{S2}) for a system of orthonormal functions as follows.  	\begin{theorem}\label{Higher dim1}
		Let $P$ be a real elliptic polynomial of $m\geq2$ or $P(\xi)=\sum\limits_{j=1}^d P_j(\xi_j)$, where $P_j$ is an arbitrary real polynomial of one variable with degree $P_j\geq2$ for all $1\leq j\leq d$. Assume $\nu\in [\frac{2}{d+2},\frac{2}{d+1})$ and $(p,q)=\left(\frac{1}{1-\nu},\frac{2}{d\nu}\right)$. For any (possibly infinite) orthonormal system $\{f_j\}_{j\in J}$ in $L^2(\mathbb{R}^d)$ and any sequence $\{n_j\}_{j\in J}$ in $\mathbb{C}$, we have
		\begin{equation*}
			\left\|\sum_{j\in J}n_j|\mathcal{W}_\frac{\nu}{2}(t)f_j|^2\right\|_{L^q(\mathcal{\mathbb{R}},L^p(\mathbb{R}^d))} \leq C\|\{n_j\}_{j\in J}\|_{\ell^\frac{2p}{p+1}},
		\end{equation*}
		where $C>0$ is independent of $\{f_j\}_{j\in J}$ and $\{n_j\}_{j\in J}$.
	\end{theorem}
	For the higher dimension, we also can define the restriction problem (in a similar fashion we discussed for the real line) and in this scenario, Theorem  \ref{Higher dim1} can also be seen as an orthonormal restriction theorem for the Fourier extension operator on noncompact hypersurfaces of the form $S=\{(\xi, P(\xi)) : \xi \in \mathbb{R}^{d}  \}$, where  $P$ be a real elliptic polynomial of $m\geq2$ or $P(\xi)=\sum\limits_{j=1}^d P_j(\xi_j)$ and $P_j$'s  are arbitrary real polynomial of one variable with degree $P_j\geq2$ for all $1\leq j\leq d$.

	Apart from the Introduction, the presentation of the paper is the following. 
	\begin{itemize}
		\item In Section \ref{sec2}, we recall complex interpolation in Schatten spaces and a duality principle, which will play a crucial role in order to prove   Strichartz estimates from a single function to an orthonormal system.
		\item In Section \ref{sec3}, we again recall Strichartz inequality for a single function associated with a general self-adjoint operator $L$ on a measure space $X$.
		\item In Section \ref{sec4}, we extend Strichartz estimates from a single function to a system of orthonormal initial data. We also investigate inhomogeneous orthonormal Strichartz estimates. 
		\item In Section \ref{sec5}, we provide some particular examples for the self-adjoint operator $L$ such as  Hermite operators, twisted Laplacians, Laguerre operators, and generalized  $(k, a)$ operator and obtain orthonormal Strichartz estimates.     
		\item In Section \ref{sec6},  we define Hartree equation and investigate global well-posedness for the Hartree equation in Schatten spaces can be obtained as an application of the orthonormal Strichartz inequalities.
		\item Section \ref{sec7} is devoted to study extend the smoothing type estimates proved by Kenig-Ponce-Vega in \cite{Kenig-Ponce-Vega} for a system of orthonormal initial data in one and higher dimensions.  
	\end{itemize}
	
	\section{Preliminaries}\label{sec2}
	\subsection{Schatten class}
	In this subsection, we recall the definition of Schatten spaces. For a detailed study of these spaces, we refer to see \cite{Simon}. Let $\mathcal{H}$ be a complex and separable Hilbert space. Let $T:  \mathcal{H}\rightarrow \mathcal{H}$ be a compact operator and $T^*$ denote the adjoint of $T$. The singular values of $T$ are the non-zero eigenvalues of $|T|:=\sqrt{T^*T}$, which form an at most countable set denoted by $\{s_j(T)\}_{j\in\mathbb{N}}$. For $\lambda\in[1,\infty)$, the Schatten space $\mathfrak{S}^\lambda(\mathcal{H})$ is defined as the space of all compact operators $T$ on $\mathcal{H}$ such that
	\begin{equation*}
		\left(\sum_{j=1}^{\infty}s_j(T)^\lambda\right)^{\frac{1}{\lambda}}<\infty.
	\end{equation*}
	Then the space $\mathfrak{S}^\lambda(\mathcal{H})$ is a Banach space endowed with the Schatten $\lambda$-norm defined by
	\begin{equation*}
		\|T\|_{\mathfrak{S}^\lambda(\mathcal{H})}=\left(\sum_{j=1}^{\infty}s_j(T)^\lambda\right)^{\frac{1}{\lambda}},\, \forall~ T\in \mathfrak{S}^\lambda(\mathcal{H}).
	\end{equation*}
	In general, we denote the space  $\mathfrak{S}^\infty(\mathcal{H})$,  as the space of the compact operators rather than the bounded operators 
	but the norm is the same as for the bounded operators. In other words, for $T\in \mathfrak{S}^\infty(\mathcal{H})$, we denote $\|T\|_{\mathfrak{S}^\infty(\mathcal{H})}$ as the operator norm of $T$. Moreover, we have the property that  $\left(\mathfrak{S}^\infty(\mathcal{H})\right)^*=\mathfrak{S}^1(\mathcal{H})$ and $\left(\mathfrak{S}^1(\mathcal{H})\right)^*=\mathfrak{B}(\mathcal{H},\mathcal{H})$, where $\mathfrak{B}(\mathcal{H},\mathcal{H})$ is the space of all bounded operators on $\mathcal{H}$. Note that, this property is different from the $L^p$ spaces. In particular, when $\lambda=1$, an operator belonging to the space $\mathfrak{S}^1(\mathcal{H})$ is known as a trace class operator; when $\lambda=2$, an operator belonging to the space $\mathfrak{S}^2(\mathcal{H})$ is known as a Hilbert-Schmidt operator. In the proof of our Strichartz estimates for orthonormal systems of initial data, we will see that Schatten spaces arise naturally via   duality argument. 
	
	\subsection{Complex interpolation in Schatten space}\label{Stein} Let $a_0, a_1\in \R$  such that $a_0<a_1$. A family of operators $\{T_z\}$ on $X\times \mathcal{I}$ defined in a strip $a_0\leq Re(z)\leq a_1$ in the complex plane is analytic in the sense of Stein if it has the following properties:
	
	$(1)$ For each $z: a_0\leq Re(z)\leq a_1$, $T_z$ is a linear transformation of simple functions on $X\times \mathcal{I}$.
	
	$(2)$ For all simple functions $F, G$ on $\mathcal{I}\times X$, the map $z\rightarrow \langle G, T_zF\rangle
	$ is analytic in $a_0<Re (z)<a_1$ and continuous in $a_0\leq Re (z)\leq a_1$.
	
	$(3)$ Moreover, $\sup_{a_0\leq\lambda\leq a_1}\left|\langle G, T_{\lambda+is}F\rangle\right|\leq C(s)$ for some $C(s)$ with at most a (double) exponential growth in $s$.
	
	The following complex interpolation result in Schatten space will be used frequently in order to prove our main results.  For a detailed study on complex interpolation theory, one can see   \cite[Proposition 1]{FS}, \cite[Theorem 2.9]{Simon}, and  \cite[Theorem 2.26]{Nguyen}.
	\begin{lemma}\label{Schatten-bound}
		Let $\{T_z\}$ be an analytic family of operators on $X\times \mathcal{I}$ in the sense of Stein defined in the strip $a_0\leq Re (z) \leq a_1$ for some $a_0<a_1$. If there exist $M_0, M_1, b_0, b_1>0$, $1\leq p_0,p_1,q_0,q_1\leq \infty$, and $1\leq r_0,r_1<\infty$ such that    for all simple functions $W_1,W_2$ on $X\times \mathcal{I}$  and  $s\in\mathbb{R}$, we have 
		\begin{equation*}		
			\begin{aligned}
				&
				\left\|W_1T_{a_0+is}W_2\right\|_{\mathfrak{S}^{r_0}\left(L^2(\mathcal{I}, L^2(X))\right)}\leq M_0e^{b_0|s|}\left\|W_1\right\|_{L^{q_0}(\mathcal{I}, L^{p_0}(X))}\left\|W_2\right\|_{L^{q_0}(\mathcal{I}, L^{p_0}(X))}	\end{aligned}
		\end{equation*}	
		and \begin{equation*}
			\begin{aligned}\left\|W_1T_{a_1+is}W_2\right\|_{\mathfrak{S}^{r_1}\left(L^2(\mathcal{I}, L^2(X))\right)}\leq M_1e^{b_1|s|}\left\|W_1\right\|_{L^{q_1}(\mathcal{I}, L^{p_1}(X))}\left\|W_2\right\|_{L^{q_1}(\mathcal{I}, L^{p_1}(X))}.
			\end{aligned}
		\end{equation*}
		Then for all $\theta\in [0,1]$, we have 
		\begin{equation*}
			\left\|W_1T_{a_\theta}W_2\right\|_{\mathfrak{S}^{r_\theta}\left(L^2(\mathcal{I}, L^2(X))\right)}\leq  M_0^{1-\theta}M_1^\theta\left\|W_1\right\|_{L^{q_\theta}(\mathcal{I}, L^{p_\theta}(X))}\left\|W_2\right\|_{L^{q_\theta}(\mathcal{I}, L^{p_\theta}(X))},
		\end{equation*}
		where $a_\theta, r_\theta, p_\theta$ and $q_\theta$ are defined by
		\begin{equation*}
			a_\theta=(1-\theta)a_0+a_1,\; \frac{1}{r_\theta}=\frac{1-\theta}{r_0}+\frac{\theta}{r_1},\;\frac{1}{p_\theta}=\frac{1-\theta}{p_0}+\frac{\theta}{p_1}\;\text{and}\;\frac{1}{q_\theta}=\frac{1-\theta}{q_0}+\frac{\theta}{q_1}.
		\end{equation*}
		Also, when one of the exponents $r_j=\infty$,  for $j=0$ or $j=1$, then the $\mathfrak{S}^\infty$-norm on the left side of the two previous bounds can be replaced by the usual operator norm.
	\end{lemma}
	\subsection{Duality principle}
	In order to deduce estimates from a single function to an orthonormal system, the duality principle lemma in Schatten spaces plays a crucial role in our context. We refer to \cite[Lemma 3]{FS} with appropriate modifications to obtain the following duality principle.
	\begin{lemma}[Duality principle]\label{duality-principle} Assume that $A$ is a bounded operator from $L^2(X)$ to $L^{q^\prime}(\mathcal{I}, L^{p^\prime}(X))$ for some $1\leq p,q\leq 2$. Then the following statements are equivalent:
		
		$(1)$ There is a constant $C>0$ such that for all $W\in L^{\frac{2q}{2-q}}(\mathcal{I}, L^{\frac{2p}{2-p}}(X))$
		\begin{equation*}
			\big\|WAA^*\overline{W}\big\|_{\mathfrak{S}^\lambda(L^2(\mathcal{I}, L^2(X)))}\leq C\big\|W\big\|_{L^{\frac{2q}{2-q}}(\mathcal{I}, L^{\frac{2p}{2-p}}(X))}^2,
		\end{equation*}
		where the function $W$ is interpreted as an operator which acts by multiplication.
		
		$(2)$ There is a constant $C'>0$ such that for any orthonormal system $\{f_j\}_{j\in J}$ in $L^2(X)$ and any sequence $\{n_j\}_{j\in J}$ in $\mathbb{C}$
		\begin{equation*}
			\bigg\|\sum_{j\in J}n_j|Af_j|^2\bigg\|_{L^{\frac{q^\prime}{2}}(\mathcal{I}, L^{\frac{p^\prime}{2}}(X))}\leq C'\|\{n_j\}_{j\in J}\|_{\ell^{\lambda^\prime}}.
		\end{equation*}		
	\end{lemma}
	We can also extend the duality principle in Lorentz spaces with trivial
	modifications as given in \cite{BHLNS}. Before that, first, we recall the definition of Lorentz spaces on the measure space $(X, \mu)$. For a  function $f$ on $X$, the decreasing rearrangement of $f$ is the function $f^*$ defined on $[0, \infty)$ by 
	$$f^*(s)=\inf\{\alpha>0: d_f(\alpha)\leq s\},$$
	where  $$d_f(\alpha)=\mu(\{x\in X: |f(x)|>\alpha\}).$$
	Then for $0<p, q\leq \infty,$ we denote the space $L^{p, q}(X, \mu)$ consist all such measurable functions on $X$ such that $\|f\|_{L^{p, q}}<\infty$, where 
	$$\|f\|_{L^{p, q}}=\begin{cases}
		\left(\int_0^\infty (s^{\frac{1}{p}} f^*(s))^q \frac{ds}{s}\right)^{\frac{1}{ q}},  \quad 
		\text{if} ~~q<\infty,\\
		\sup_{s>0} s^{\frac{1}{p}} f^*(s),   \hspace{1.4cm}
		\text{if} ~~q=\infty.
	\end{cases}$$
	One can observe that $L^{p, p}=L^p$  and $L^{\infty, \infty}=L^\infty$. The following remark is the duality principle in Lorentz spaces.

	\begin{remark} \label{Duality-Lorentz}Assume that $A$ is a bounded operator from $L^2(X)$ to $L^{q^\prime,r^\prime}(\mathcal{I}, L^{p^\prime}(X))$ for some $1\leq p,q,r\leq 2$. Then the following statements are equivalent:
		
		$(1)$ There is a constant $C>0$ such that for all $W\in L^{\frac{2q}{2-q},\frac{2r}{2-r}}(\mathcal{I}, L^{\frac{2p}{2-p}}(X))$
		\begin{equation*}	\big\|WAA^*\overline{W}\big\|_{\mathfrak{S}^\lambda(L^2(\mathcal{I}, L^2(X)))}\leq C\big\|W\big\|_{L^{\frac{2q}{2-q},\frac{2r}{2-r}}(\mathcal{I}, L^{\frac{2p}{2-p}}(X))}^2.
		\end{equation*}

		$(2)$ There is a constant $C'>0$ such that for any orthonormal system $\{f_j\}_{j\in J}$ in $L^2(X)$ and any sequence $\{n_j\}_{j\in J}$ in $\mathbb{C}$
		\begin{equation*}
			\bigg\|\sum_{j\in J}n_j|Af_j|^2\bigg\|_{L^{\frac{q^\prime}{2},\frac{r^\prime}{2}}(\mathcal{I}, L^{\frac{p^\prime}{2}}(X))}\leq C'\|\{n_j\}_{j\in J}\|_{\ell^{\lambda^\prime}}.
		\end{equation*}		
	\end{remark}
	
	\section{Strichartz inequality for a single function}\label{sec3}
	In this section, we establish the Strichartz inequality in Lebesgue spaces as well as in Lorentz spaces for a single function. 
	
	Consider  the  Cauchy problem for Schr\"{o}dinger equation associated with $L$ as
	\begin{equation}\label{Cauchy}
		\begin{cases}
			i\partial_t u(x,t)+L u(x,t)=F(x,t)\\
			u(0,t)=f(x).
		\end{cases}
	\end{equation}
	Then, by Duhamel's principle, the solution of the Cauchy problem \eqref{Cauchy} is formally given by 
	\begin{equation*}
		u(x,t)=e^{itL}f(x)-i\int_0^t e^{i(t-s)L}F(x,s)ds.
	\end{equation*}
	The following result is known as   Strichartz estimates for Schr\"{o}dinger equation associated with $L$ on $X.$  \begin{theorem}\label{Strichartz-single}
		Assume $L$ is a non-negative self-adjoint operator on $L^2(X)$ and satisfies Assumption $(\mathcal{A})$ and $(\mathcal{B})$. Then for any $p_i,q_i\geq 2$ $(i=1,2)$ such that
		\begin{equation*}
			\frac{2}{q_i}+\frac{n}{p_i}=\frac{n}{2},
		\end{equation*}
		except for $(q_i,p_i,n)=(2,\infty,2)$,
		we have\\
		$(1)$ (The homogeneous Strichartz estimate)
		\begin{equation*}
			\|e^{itL}f\|_{L^{q_1}\left(\mathcal{I},L^{p_1}(X)\right)}\lesssim \|f\|_{L^2(X)}.
		\end{equation*}
		$(2)$ (The dual homogeneous Strichartz estimate)
		\begin{equation*}
			\left\|\int_\mathcal{I}e^{-isL}F(\cdot,s)ds\right\|_{L^2(X)}\lesssim \|F\|_{L^{q_1'}\left(\mathcal{I},L^{p_1'}(X)\right)}.
		\end{equation*}
		$(3)$ (The retarded Strichartz estimate)
		\begin{equation*}
			\left\|\int_0^te^{i(t-s)L}F(\cdot,s)ds\right\|_{L^{q_2}\left(\mathcal{I},L^{p_2}(X)\right)}\lesssim \|F\|_{L^{q_1'}\left(\mathcal{I},L^{p_1'}(X)\right)}.
		\end{equation*}
	\end{theorem}
	The procedure to prove the above theorem is classical and some good references are given, for example, by Ginibre-Velo \cite{GV} and Keel-Tao \cite{KT}. For the reader's convenience, we just provide the main steps of the proof in Section \ref{sec8} for the non-endpoints Strichartz estimates. Indeed, we can also prove endpoints Strichartz estimates 
	following procedures as given in  \cite{KT}.  
	
	\begin{remark}\label{Strichartz-Lorentz}
		One of the key ingredients to obtain Strichartz estimates is the Hardy-Littlewood-Sobolev inequality for Lebesgue functions.  However, in order to prove Strichartz inequality  in  
		Lorentz space, we need a refined version of the Hardy-Littlewood-Sobolev inequality for Lebesgue functions. From   \cite{Neil}, we have the following  Littlewood-Sobolev inequality in  
		Lorentz space:  
		
		Let  $\sigma\in (0,1)$  and $p_1,p_2,r_1,r_2\in (1,+\infty)$ such that 
		\begin{equation*}
			\frac{1}{q_1}+\frac{1}{q_2}+\sigma=2 ~~\text{ and } ~~\frac{1}{r_1}+\frac{1}{r_2}\geq1,
		\end{equation*}
		then  
		\begin{equation*}
			\left|\int_\mathcal{I}\int_\mathcal{I}\frac{g(t)h(s)}{|t-s|^\sigma}dsdt\right|\lesssim \|g\|_{L^{q_1,r_1}(\mathcal{I})}\|h\|_{L^{q_2,r_2}(\mathcal{I})}.
		\end{equation*}
	\end{remark}
	Now applying the same procedure as in the proof of  Theorem \ref{Strichartz-single} with  the Littlewood-Sobolev
	inequality in Lorentz space,  we have the following Strichartz estimates for functions in Lorentz spaces.  
	\begin{theorem}\label{Lorentz-single}
		Let  $L$ is a non-negative self-adjoint operator on $L^2(X)$ and satisfies Assumption $(\mathcal{A})$ and $(\mathcal{B})$. Then for any $q,p,r\geq 2$   such that
		\begin{equation*}
			\frac{2}{q}+\frac{n}{p}=\frac{n}{2},
		\end{equation*}
		except for $(q,p,n)=(2,\infty,2)$,  then 
		\begin{equation*}
			\|e^{itL}f\|_{L^{q,r}(\mathcal{I},L^p(X))}\lesssim \|f\|_{L^2(X)}.
		\end{equation*}
	\end{theorem} 
	\begin{corollary}
		Under the same assumptions as in  Theorem \ref{Strichartz-single}, for $f\in L^2(X)$ and $F\in L^{q_2'}\left(\mathcal{I},L^{p_1'}(X)\right)$,     the solution $u$ of the solution of the Schr\"{o}dinger equation \eqref{Cauchy} belongs to $L^{q_2}\left(\mathcal{I},L^{p_2}(X)\right)$ and satisfies
		\begin{equation*}
			\|u\|_{L^{q_2}\left(\mathcal{I},L^{p_2}(X)\right)}\lesssim \|f\|_{L^2(X)}+\|F\|_{L^{q_2'}\left(\mathcal{I},L^{p_1'}(X)\right)}.
		\end{equation*}
		
	\end{corollary}
	
	\section{Strichartz inequalities for system of orthonormal functions}\label{sec4}
	This section is devoted to extending Theorem \ref{Strichartz-single} part (1) and Theorem \ref{Lorentz-single} for a single function to a system of orthonormal functions.   Note that $Af(x,t)=e^{itL}f(x)$ is the solution to the linear Schr\"odinger equation (\ref{Cauchy}). Also,  the solution $Af(x,t)=e^{itL}f(x)$  is a bounded operator from $L^2(X)$ to $L^q(\mathcal{I},L^p(X))$ and its adjoint operator $A^*$,  from $L^{q'}(\mathcal{I},L^{p'}(X))$ to $L^2(X)$ is defined by
	\begin{equation*}
		A^*F(x)=\int_\mathcal{I} e^{-isL}F(x,s)ds.
	\end{equation*}
	We define $\mathcal{T}=AA^*$ by
	\begin{equation*}
		\mathcal{T}F(x,t)=\int_\mathcal{I} e^{i(t-s)L}F(x,s)ds=\int_\mathcal{I}\int_XK_{i(t-s)}(x,y)F(y,s)d\mu(y)ds.
	\end{equation*}
	Then we have the following result regarding the boundedness of $\mathcal{T}$ in Schatten spaces. 
	\begin{lemma}\label{Schatten}
		Assume $L$ is a non-negative self-adjoint operator on $L^2(X)$ satisfying $(\mathcal{A})$ and $(\mathcal{B})$. Then for any $p,q\geq 1$ such that
		\begin{equation*}
			\frac{2}{q}+\frac{n}{p}=1 \quad \text{and}\quad n+1<p\leq n+2, 
		\end{equation*}
		we have
		\begin{equation*}
			\big\|W_1\mathcal{T}W_2\big\|_{\mathfrak{S}^p(L^2(\mathcal{I}, L^2(X)))}\leq C\big\|W_1\big\|_{L^q(\mathcal{I}, L^p(X))}\big\|W_2\big\|_{L^q(\mathcal{I}, L^p(X))},
		\end{equation*} 
		for all $W_1,W_2\in L^q(\mathcal{I}, L^p(X))$, where $C>0$ is independent of $W_1$ and $W_2$.
	\end{lemma}
	\begin{proof}
		For any $\varepsilon>0$, we define
		\begin{equation*}
			T_\varepsilon F(x,t)=\int_X\int_\mathcal{I} \mathcal{K}_\varepsilon (x,y,t-s)F(y,s)dsd\mu(y),
		\end{equation*}
		where $\mathcal{K}_\varepsilon (x,y,t)=\chi_{\varepsilon< |t|< T_0}(t)K_{it}(x,y)$. Then for any $t\in \mathcal{I}$, it follows from $(\mathcal{B})$ that  		\begin{equation}\label{varepsilon-decay}
			\sup_{x,y\in X}|\mathcal{K}_\varepsilon (x,y,t)|\lesssim |t|^{-\lambda_0},\;\text{where}\; \lambda_0=\frac{n}{2}.
		\end{equation}
		To use Stein's analytic complex interpolation, for $z\in\mathbb{C}$ with $Re(z)\in [-1,\lambda_0]$, we define
		\begin{equation*}
			\mathcal{K}_{z,\varepsilon} (x,y,t)=t^z\mathcal{K}_\varepsilon (x,y,t)
		\end{equation*}
		and define the corresponding operator as
		\begin{equation*}
			T_{z,\varepsilon} F(x,t)=\int_X\int_\mathcal{I} \mathcal{K}_{z,\varepsilon} (x,y,t-s)F(y,s)dsd\mu(y).
		\end{equation*}
		Then for any $\varepsilon>0$,   $\{T_{z,\varepsilon}\}_z$ forms an analytic family of operators in the sense of Stein (see 
		Subsection \ref{Stein}). Moreover, from \eqref{varepsilon-decay}, we can deduce that 
		\begin{equation}\label{z-varepsilon-decay}
			\sup_{x,y\in X}|\mathcal{K}_{z,\varepsilon} (x,y,t)|\lesssim |t|^{Re(z)-\lambda_0},\; \forall t\in \mathcal{I}.
		\end{equation}
		Now keeping in mind  that the $\mathfrak{S}^2$-norm is the  usual Hilbert-Schmidt norm, using the Hardy-Littlewood-Sobolev inequality in Lebesgue spaces  along with \eqref{z-varepsilon-decay}, we obtain 
		\begin{align*}
			\big\|W_1 T_{z,\varepsilon} W_2\big\|^2_{\mathfrak{S}^2(L^2(\mathcal{I}, L^2(X)))}&=\int_{\mathcal{I}\times\mathcal{I}}\int_{X\times X}|W_1(x,t)|^2|\mathcal{K}_{z,\varepsilon} (x,y,t-s)|^2|W_2(y,s)|^2d\mu(x)d\mu(y)dtds\\
			&\lesssim \int_{\mathcal{I}\times\mathcal{I}}\int_{X\times X}\frac{|W_1(x,t)|^2|W_2(y,s)|^2}{|t-s|^{2\lambda_0-2Re(z)}}d\mu(x)d\mu(y)dtds\\
			&=\int_{\mathcal{I}\times\mathcal{I}}\frac{\|W_1(\cdot,t)\|_{L^2(X)}^2\|W_2(\cdot,s)\|_{L^2(X)}^2}{|t-s|^{2\lambda_0-2Re(z)}}dtds\\
			&\lesssim \|W_1\|^2_{L^{2\tilde{u}}\left(\mathcal{I},L^2(X)\right)}\|W_2\|^2_{L^{2\tilde{u}}\left(\mathcal{I},L^2(X)\right)},
		\end{align*}
		provided that $0\leq 2\lambda_0-2Re(z)<1$ and $\frac{2}{\tilde{u}}+2\lambda_0-2Re(z)=2$. If we write $u=2\tilde{u}$, then $\frac{1}{u}\in (\frac{1}{4},\frac{1}{2}]$ and
		\begin{equation}\label{Interpolation1}
			\big\|W_1 T_{z,\varepsilon} W_2\big\|_{\mathfrak{S}^2(L^2(\mathcal{I}, L^2(X)))}\lesssim \|W_1\|_{L^u\left(\mathcal{I},L^2(X)\right)}\|W_2\|_{L^u\left(\mathcal{I},L^2(X)\right)},
		\end{equation}
		provided that $\frac{1}{u}=\frac{1}{2}-\frac{1}{2}(\lambda_0-Re(z))$ and $Re(z)\in (\frac{2\lambda_0-1}{2},\lambda_0]$.
		
		On the other hand,  for $Re(z)=-1$, our aim is to show that the operator $T_{z,\varepsilon}: L^2(\mathcal{I}, L^2(X))\rightarrow L^2(\mathcal{I}, L^2(X))$ is bounded with some constant depending only on $n$ and $Im(z)$ exponentially. In fact, we can write
		\begin{align*}
			T_{z,\varepsilon} F(x,t)&=\int_X\int_\mathcal{I} \mathcal{K}_{z,\varepsilon} (x,y,t-s)F(y,s)dsd\mu(y)\\
			&=\int_{\varepsilon< |t|< T_0} s^{-1+iIm(z)} ds \int_X K_{is}(x,y)F(y,t-s)dy\\
			&=\int_{\varepsilon< |t|< T_0} s^{-1+iIm(z)} e^{isL}F(x,t-s)ds.
		\end{align*}
		Since   $e^{-itL}$ is  an     unitary operator on $L^2(X)$, it follows that  
		\begin{align*}
			\|T_{z,\varepsilon} F\|_{L^2(X)}&=\|e^{-itL}T_{z,\varepsilon} F\|_{L^2(X)}\\
			&=\left\|\int_{\varepsilon< |t|< T_0} s^{-1+iIm(z)} G(\cdot,t-s)ds\right\|_{L^2(X)},
		\end{align*}
		where $G(x,t)=e^{-itL}F(x,t)$ and
		\begin{equation}\label{norm}
			\|G(\cdot,t)\|_{L^2(X)}=\|F(\cdot,t)\|_{L^2(X)}.  
		\end{equation} If we define
		\begin{equation*}
			H_{z,\varepsilon}: h(t)\mapsto \int_{\varepsilon< |t|< T_0} s^{-1+iIm(z)} h(t-s)ds,
		\end{equation*}
		then we have
		\begin{equation}\label{norm1}
			\|T_{z,\varepsilon} F\|_{L^2(\mathcal{I},L^2(X))}=\|H_{z,\varepsilon} G\|_{L^2(\mathcal{I},L^2(X))}=\|H_{z,\varepsilon} G\|_{L^2(X,L^2(\mathcal{I}))}.
		\end{equation}
		Since the operator $H_{z,\varepsilon}$ is just Hilbert transform up to $iIm(z)$, we have the bound $H_{z,\varepsilon}: L^2(\mathcal{I})\mapsto L^2(\mathcal{I})$ with some constant depending only on $Im(z)$ exponentially. For a detailed analysis, we refer to   Vega \cite{Vega}. Now  combining   \eqref{norm} and \eqref{norm1}, we obtain
		\begin{align*}
			\|T_{z,\varepsilon} F\|_{L^2(\mathcal{I},L^2(X))}&\leq C(Im(z))\|G\|_{L^2(X,L^2(\mathcal{I}))}\\
			&=C(Im(z))\|G\|_{L^2(\mathcal{I},L^2(X))}\\
			&=C(Im(z))\|F\|_{L^2(\mathcal{I},L^2(X))},
		\end{align*}
		which is our desired bound for $T_{z,\varepsilon}: L^2(\mathcal{I}, L^2(X))\rightarrow L^2(\mathcal{I}, L^2(X))$ when $Re(z)=-1$.
		
		Noting the fact that $\mathfrak{S}^\infty$-norm is the usual operator norm, for $Re(z)=-1$, we get 
		\begin{align}\label{Interpolation2}\nonumber
			\|W_1T_{z,\varepsilon}W_2\|_{\mathfrak{S}^\infty(L^2(\mathcal{I},L^2(X)))}&= \|W_1T_{z,\varepsilon}W_2\|_{L^2(\mathcal{I},L^2(X))\rightarrow L^2(\mathcal{I},L^2(X))}\\\nonumber
			&\leq \|W_1\|_{L^\infty(\mathcal{I},L^\infty(X))}\|T_{z,\varepsilon}\|_{L^2(\mathcal{I},L^2(X))\rightarrow L^2(\mathcal{I},L^2(X))}\|W_2\|_{L^\infty(\mathcal{I},L^\infty(X))}\\ 
			&\leq C(Im(z)) \|W_1\|_{L^\infty(\mathcal{I},L^\infty(X))} \|W_2\|_{L^\infty(\mathcal{I},L^\infty(X))}.
		\end{align}
		Finally, applying the complex interpolation method given in  Lemma \ref{Schatten-bound}, between (\ref{Interpolation1}) and (\ref{Interpolation2}) for $z=0$,  we get 
		\begin{equation}\label{Interpolation3}
			\big\|W_1 T_\varepsilon W_2\big\|_{\mathfrak{S}^p(L^2(\mathbb{I}, L^2(X)))}\leq C\big\|W_1\big\|_{L^q(\mathbb{I}, L^p(X))}\big\|W_2\big\|_{L^q(\mathbb{I}, L^p(X))},
		\end{equation} 
		for some $C>0$ independent of $\varepsilon$ as long as
		\begin{equation*}
			\frac{1}{q}+\frac{\lambda_0}{p}=\frac{1}{2},\; \text{with}\; 2\lambda_0+1<p\leq 2(\lambda_0+1).
		\end{equation*}
		Finally,     letting $\varepsilon\rightarrow 0^+$ in (\ref{Interpolation3}), we complete the proof of the theorem. 
	\end{proof}
	Using the above theorem,  in the following result,  we establish the orthonormal Strichartz estimate in Lebesgue spaces, i.e., orthonormal version of Theorem \ref{Strichartz-single} part (1).  \begin{theorem}\label{OSI}
		Assume $L$ is a non-negative self-adjoint operator on $L^2(X)$ satisfying Assumption $(\mathcal{A})$ and $(\mathcal{B})$. Suppose $p,q\geq 1$ satisfy \begin{equation*}
			\frac{2}{q}+\frac{n}{p}=n. 
		\end{equation*}
		
		(i) If  $  1\leq p<\frac{n+1}{n-1},$  then 
		\begin{equation}\label{eq1}
			\left\|\sum_{j\in J}n_j|e^{itL}f_j|^2\right\|_{L^q(\mathcal{I},L^p(X))} \leq C\left\|\{n_j\}_{j\in J}\right\|_{\ell^\frac{2p}{p+1}},
		\end{equation}
		holds for any (possibly infinite) orthonormal system $\{f_j\}_{j\in J}$ in $L^2(X)$ and any sequence $\{n_j\}_{j\in J}$ in $\mathbb{C}$,
		where $C>0$ is independent of $\{f_j\}_{j\in J}$ and $\{n_j\}_{j\in J}$.\\
		
		(ii) If 
		\begin{align*}
			\begin{cases}  1\leq n<2, \\
				\frac{n+1}{n-1}\leq p\leq \infty,
			\end{cases}or ~~ \begin{cases}  n=2  \\
				3\leq p<\infty,
			\end{cases} or ~~ \begin{cases}  n>2\\
				\frac{n+1}{n-1}\leq p\leq \frac{n}{n-2},
			\end{cases}
		\end{align*}   then for  any $1\leq\lambda<q$
		\begin{equation}\label{eq2}
			\left\|\sum_{j\in J}n_j|e^{itL}f_j|^2\right\|_{L^q(\mathcal{I},L^p(X))} \leq C\left\|\{n_j\}_{j\in J}\right\|_{\ell^\lambda},
		\end{equation}
		holds for any (possibly infinite) orthonormal system $\{f_j\}_{j\in J}$ in $L^2(X)$ and any sequence $\{n_j\}_{j\in J}$ in $\mathbb{C}$,
		where $C>0$ is independent of $\{f_j\}_{j\in J}$ and $\{n_j\}_{j\in J}$.
	\end{theorem}

	\begin{proof} 
		For the time being, assuming part $(i)$, one can notice that,  using the   triangle inequality and  the Strichartz estimate for a single function in  Part (1) of Theorem  \ref {Strichartz-single}, 
		\begin{equation}\label{Larger}
			\left\|\sum_{j\in J}n_j|e^{itL}f_j|^2\right\|_{L^q(\mathcal{I},L^p(X))} \leq  \sum_{j\in J}|n_j| \left\||e^{itL}f_j|^2\right\|_{L^q(\mathcal{I},L^p(X))} \leq C\sum_{j\in J}|n_j|,
		\end{equation}
		which gives  \eqref{eq2} for $\lambda=1$  without making use of the orthogonality of the system $\{f_j\}_{j\in J}$ in $L^2(X)$. However, by capitalizing on the orthogonality of the $f_j$,  in part $(i)$, we see that it holds for $\lambda=\frac{2p}{p+1}>1$.   Of course, the above trivial inequality \eqref{Larger} can be proved for a larger range of $p$ than the range $ 1\leq p< \frac{n+1}{n-1} $ as in part $(i)$. In fact, it holds for all such exponential pairs $(q, p)$  that satisfy part (1) of  Theorem  \ref{Strichartz-single}, i.e., Strichartz estimates for a single function associate with $L$. Particularly,  it holds for  the Keel-Tao endpoint $(q, p)=\left(1, \frac{n}{n-2}\right)$ for $n>2 $  and $(q, p)=\left(  \frac{2}{n}, \infty\right)$ for $1\leq n<2$ with $\lambda=1$.
		Now, as pointed out in \cite{FS2016}, 
		\begin{itemize}
			\item for $n>2$, the estimates in $(ii)$ come out from those in $(i)$ by interpolation between $(q,p)=(1,\frac{n}{n-2})$ (with $\lambda=1$)  and points arbitrarily close to $(q,p)=(\frac{n+1}{n},\frac{n+1}{n-1})$  (with $\lambda=\frac{2p}{p+1}$);
			\item for $n=2$, the estimates in $(ii)$ come out from those in $(i)$ by interpolation between points arbitrarily close to $(q,p)=(1,\infty)$   (with $\lambda=1$) and points arbitrarily close to $(q,p)=(\frac{n+1}{n},\frac{n+1}{n-1})$ (with $\lambda=\frac{2p}{p+1}$);  
			\item   for $1\leq n<2$, the estimates in $(ii)$ come out from those in $(i)$ by interpolation between $(q,p)=(\frac{2}{n},\infty)$  (with $\lambda=1$) and points arbitrarily close to $(q,p)=(\frac{n+1}{n},\frac{n+1}{n-1})$ (with $\lambda=\frac{2p}{p+1}$). 
		\end{itemize}   From the above discussion, in order to prove part $ (ii),$ it's enough to prove  part $(i)$. From the fact that the operator $e^{itL}$ is unitary on $L^2(X)$, it follows that
		\begin{equation*}
			\left\|\sum_{j\in J}n_j|e^{itL}f_j|^2\right\|_{L^\infty(\mathcal{I},L^1(X))} \leq \sum_{j\in J}|n_j|,
		\end{equation*}
		for any (possibly infinite) orthonormal system $\{f_j\}_{j\in J}$ in $L^2(X)$ and any sequence $\{n_j\}_{j\in J}$ in $\mathbb{C}$. That means equation (\ref{eq1}) holds for $(q, p)=(\infty, 1)$ and  by  Lemma \ref{duality-principle}, equivalently, we can say that  the operator
		\begin{equation*}
			W\in L^2(\mathcal{I},L^\infty(X))\mapsto W\mathcal{T}\overline{W} \in\mathfrak{S}^\infty(L^2(\mathcal{I},L^\infty(X)))
		\end{equation*}
		is bounded. Again, by  Theorem \ref{Strichartz-single}, Lemma \ref{duality-principle} and \ref{Schatten}, the operator 
		\begin{equation*}
			W\in L^{n+2}(\mathcal{I},L^{n+2}(X))\mapsto W\mathcal{T}\overline{W} \in\mathfrak{S}^{n+2}(L^2(\mathcal{I},L^2(X)))
		\end{equation*}
		is also bounded. Applying the complex interpolation method in  \cite[Chapter 4]{BL}, the operator
		\begin{equation*}
			W\in L^q(\mathcal{I},L^p(X))\mapsto W\mathcal{T}\overline{W} \in\mathfrak{S}^p(L^2(\mathcal{I},L^2(X)))
		\end{equation*}
		is bounded provided that $p,q\geq 1$ such that
		\begin{equation*}
			\frac{2}{q}+\frac{n}{p}=n,\;\text{with}\; n+2\leq p\leq\infty.
		\end{equation*}
		Combined with Lemma \ref{Schatten}, we have 
		\begin{equation*}
			\big\|W_1\mathcal{T}W_2\big\|_{\mathfrak{S}^p(L^2(\mathbb{I}, L^2(X)))}\leq C\big\|W_1\big\|_{L^q(\mathbb{I}, L^p(X))}\big\|W_2\big\|_{L^q(\mathbb{I}, L^p(X))},
		\end{equation*} 
		holds true for any $p,q\geq 1$ satisfying
		\begin{equation*}
			\frac{2}{q}+\frac{n}{p}=n,\;\text{with}\; n+1< p\leq\infty.
		\end{equation*}
		Finally, by the duality principle Lemma \ref{duality-principle}, we obtain our desired result $(i)$.
	\end{proof}
	
	\begin{remark} Using the duality principle in Lorentz spaces in Remark \ref{Duality-Lorentz}, the Hardy-Littlewood-Sobolev inequality, and Strichartz estimates for a single function in Lorentz spaces in Remark \ref{Strichartz-Lorentz}, argued similarly as the proof of Theorem \ref{OSI}, we can also prove orthonormal Strichartz versions in Lorentz spaces,  which is in fact crucial to obtain the orthonormal Strichartz estimates of initial data with regularity for the classical Schr\"odinger operator in \cite{BHLNS} when $L$ is the classical Laplacian operator.
		
		Assume $L$ is a non-negative self-adjoint operator on $L^2(X)$ satisfying Assumption $(\mathcal{A})$ and $(\mathcal{B})$. Suppose $p,q,r\geq 1$ satisfies
		\begin{equation*}
			\frac{2}{q}+\frac{n}{p}=n.
		\end{equation*}
		(i) If $1\leq p<\frac{n+1}{n-1}$, then 
		\begin{equation*}
			\left\|\sum_{j\in J}n_j|e^{itL}f_j|^2\right\|_{L^{q,r}(\mathcal{I},L^p(X))} \leq C\left\|\{n_j\}_{j\in J}\right\|_{\ell^\frac{2p}{p+1}},
		\end{equation*}
		holds for any (possibly infinite) orthonormal system $\{f_j\}_{j\in J}$ in $L^2(X)$ and any sequence $\{n_j\}_{j\in J}$ in $\mathbb{C}$,
		where $C>0$ is independent of $\{f_j\}_{j\in J}$ and $\{n_j\}_{j\in J}$.\\
		(ii) If $1\leq n<2$ and $\frac{n+1}{n-1}\leq p\leq \infty$, $n=2$ and $3\leq p<\infty$, or $n>2$ and $\frac{n+1}{n-1}\leq p\leq \frac{n}{n-2}$, then for $1\leq\lambda<q$
		\begin{equation*}
			\left\|\sum_{j\in J}n_j|e^{itL}f_j|^2\right\|_{L^{q,r}(\mathcal{I},L^p(X))} \leq C\left\|\{n_j\}_{j\in J}\right\|_{\ell^\lambda},
		\end{equation*}
		holds for any (possibly infinite) orthonormal system $\{f_j\}_{j\in J}$ in $L^2(X)$ and any sequence $\{n_j\}_{j\in J}$ in $\mathbb{C}$,
		where $C>0$ is independent of $\{f_j\}_{j\in J}$ and $\{n_j\}_{j\in J}$.
	\end{remark}
	
	\begin{remark} Indeed, our method is general and can also be applied to obtain orthonormal Strichartz estimates for a general class of dispersive semigroup associated with $L$. Let $\phi: \mathbb{R}^+\rightarrow \mathbb{R}$ be smooth and $\psi\in C_c^\infty([\frac{1}{2},2])$. We consider the truncated (or called frequency localized) dispersive semigroup $U(t)=e^{it\phi(L)}\psi(\sqrt{L})$. It is easy to see that the operator obeys the energy estimate:
		\begin{equation*}
			\|U(t)f\|_{L^2(X)}\lesssim_\psi \|f\|_{L^2(X)},
		\end{equation*}
		for all $t\in\mathbb{R}$ and $f\in L^2(X)$. We ssume that
		
		\underline{Assumption $(\mathcal{A}')$} $U(t)$ is  an integral with kernel $K_{it}(x,y)$, i.e., 
		\begin{equation*}
			U(t)f(x)=\int_X K_{it}(x,y)f(y)d\mu(y), \quad   \text{a.e.}~x\in X,
		\end{equation*}
		for every $f\in L^2(X)$. 
		
		\underline{Assumption $(\mathcal{B}')$} The kernel $K_{it}(x,y)$of the operator $U(t)$ satisfies an uniform decay estimate:
		\begin{equation*}
			\sup_{x,y\in X}|K_{it}(x,y)|\lesssim_{\phi,\psi} |t|^{-\frac{n}{2}},\,|t|<T_0, \text{ for some }n\geq1,
		\end{equation*}
		where $T_0\in(0,+\infty]$. 
		
		By Keel-Tao \cite{KT}, we obtain the Strichartz estimate for $U(t)$ involving a single function. 
		\begin{equation*}
			\|U(t)f\|_{L^{2q}(\mathcal{I},L^{2p}(X))}\lesssim_{\phi,\psi} \|f\|_{L^2(X)},
		\end{equation*}
		holds for any $p,q\geq1$ satisfying $\frac{2}{q}+\frac{n}{p}=n$ except for $(q,p,n)=(1,\infty,2)$.   Argued the same as in the proof of Theorem \ref{OSI}, we establish the following orthonormal Strichartz estimates. 
		\begin{theorem}\label{truncated}
			Assume $L$ is a non-negative self-adjoint operator on $L^2(X)$ and $U(t)=e^{it\phi(L)}\psi(\sqrt{L})$ satisfies Assumption $(\mathcal{A}')$ and $(\mathcal{B}')$. Suppose $p,q\geq 1$ satisfy
			\begin{equation*}
				\frac{2}{q}+\frac{n}{p}=n.
			\end{equation*}
			
			(i) If $1\leq p<\frac{n+1}{n-1}$, then 
			\begin{equation}\label{Eq1}
				\left\|\sum_{j\in J}n_j|U(t)f_j|^2\right\|_{L^q(\mathcal{I},L^p(X))} \lesssim_{\phi,\psi} \left\|\{n_j\}_{j\in J}\right\|_{\ell^\frac{2p}{p+1}},
			\end{equation}
			holds for any (possibly infinite) orthonormal system $\{f_j\}_{j\in J}$ in $L^2(X)$ and any sequence $\{n_j\}_{j\in J}$ in $\mathbb{C}$.\\
			
			(ii) If  \begin{align*}
				\begin{cases}  1\leq n<2, \\
					\frac{n+1}{n-1}\leq p\leq \infty,
				\end{cases}or ~~ \begin{cases}  n=2  \\
					3\leq p<\infty,
				\end{cases} or ~~ \begin{cases}  n>2\\
					\frac{n+1}{n-1}\leq p\leq \frac{n}{n-2},
				\end{cases}
			\end{align*} 
			then for $1\leq\lambda<q$
			\begin{equation}\label{Eq2}
				\left\|\sum_{j\in J}n_j|U(t) f_j|^2\right\|_{L^q(\mathcal{I},L^p(X))} \lesssim_{\phi,\psi} \left\|\{n_j\}_{j\in J}\right\|_{\ell^\lambda},
			\end{equation}
			holds for any (possibly infinite) orthonormal system $\{f_j\}_{j\in J}$ in $L^2(X)$ and any sequence $\{n_j\}_{j\in J}$ in $\mathbb{C}$.
		\end{theorem}
		Indeed, these frequency localized estimates \eqref{Eq1} and \eqref{Eq2} are quite useful. For instance, when $L=-\Delta$,  the classical Laplacian on $X=\mathbb{R}^n$ and $\phi(r)=r$, Bez-Hong-Lee-Nakamura-Sawano \cite{BHLNS} proved the frequency localized estimates for the Schr\"odinger operator $e^{it\Delta}$ (see  \cite[Theorem 1.4]{BHLNS}), which was crucial to establish the orthonormal Strichartz estimates of initial data with regularity. Hence, the frequency localized estimates given in \eqref{Eq1} and \eqref{Eq2} can be regarded as a first step towards the orthonormal Strichartz estimates of initial data with regularity for a general class of dispersive semigroup associated with $L$.
	\end{remark}

	Now	consider  the    operator
	\begin{align}\label{z}
		\gamma :=\sum_{j\in J} n_{j}\left|f_{j}\right\rangle\left\langle f_{j}\right|
	\end{align} 
	which acts on $L^{2}\left(X, \mu\right) .$ Here  the  Dirac's notation \(|u\rangle\langle v|\) stands  for the
	rank-one operator $f \mapsto\langle v, f\rangle u$.  Since  the sequence \(f_{j}\) form an orthonormal system therefore,
	the scalers, \(n_{j}\) are precisely the eigenvalues of the operator $\gamma$. The evolved operator
	$$\gamma(t) :=e^{i t L} \gamma e^{-i t L}=\sum_{j\in J} n_{j}\left|e^{i t L} f_{j}\right\rangle\left\langle e^{i t L} f_{j}\right| $$
	solves (in a weak sense) the von-Neumann Schrödinger equation associated with the operator $L$ 
	$$i \dot{\gamma}(t)=[L, \gamma(t)], \quad  \gamma(0)=\gamma ,$$
	where  $[A, B]$ denotes the commutator between two operators $A$ and $B$.  Let us introduce the density  
	\begin{align}\label{y}
		\rho_{\gamma(t)} :=\sum_{j\in J} n_{j}\left|e^{i t L} f_{j}\right|^{2}.
	\end{align} 
	Then with respect to the notions defined above, the inequality   (\ref{eq1}) can be  rewritten  as
	\begin{align}\label{7}
		\left\|\rho_{\gamma(t)}\right\|_{L^{q}\left(\mathcal{I}, L^{p}\left(X\right)\right)} \leqslant C_S \|\gamma\|_{\mathfrak{S}^{\frac{2 p}{p+1}}},\end{align}
	where the norm 
	$$
	\|\gamma\|_{\mathfrak{S}^{\frac{2 p}{p+1}}} :=\left(\sum_{j\in J}\left|n_{j}\right|^{\frac{2 p}{p+1}}\right)^{\frac{p+1}{2 p}}
	$$
	is called the Schatten norm of the operator \(\gamma\).  
	
	Again, for any  multiplication operator $V$   on $L^2(X)$,  
	$V(x)\gamma(t)$ is an integral operator with kernel is given $$K(x, y)=V(x)\sum_{j\in J} n_j e^{itL}u_j(x) \overline{e^{itL}u_j(y)}.$$
	Then 
	\begin{align*}
		\operatorname{Tr}(V(x)\gamma(t))&=\int_{ {X}} K(x, x) dx\\
		&=\int_{ {X}} V(x)\sum_{j\in J} n_j \left|e^{itL}u_j(x)\right|^2  dx\\
		&=\int_{ {X}} V(x)\rho_{\gamma}(x) ~dx.
	\end{align*}  
	Moreover,  for a time-dependent potential $V(t, x) $, we therefore obtain
	$$
	\begin{aligned}
		\left|\operatorname{Tr}\left( \gamma \int_{\mathcal{I}} e^{-i t L
		} V(t, x) e^{i t L} d t\right) \right|&=\left|\int_{\mathcal{I}}  \operatorname{Tr}\left(V(t, x) e^{i t L} \gamma e^{-i t L}\right) d t\right| \\
		&=\left|\int_{\mathcal{I}}  \operatorname{Tr}\left(V(t, x)   \gamma (t)\right) d t\right| \\		& =\left|\int_{\mathcal{I}} \int_{X} V(t, x) \rho_{\gamma(t)}(x) d x d t\right|\\ &\leq\|V\|_{L^{q^{\prime}}\left(\mathcal{I}, L^{p^{\prime}}\left(X\right)\right)}\left\|\rho_{\gamma(t)}\right\|_{L^q\left(\mathcal{I}, L^p\left(X\right)\right)}.
	\end{aligned}
	$$
	Thus  by duality argument, \eqref{7}   turns out to be equivalent as
	\begin{align}\label{duali}
		\left\| \int_{\mathcal{I}} e^{-i t L}V(t, x) e^{i t L} d t\right\|_{\mathfrak{S}^{2p'}} \leq C_S \|V\|_{L^{q^{\prime}}\left(\mathcal{I}, L^{p^{\prime}}\left(X\right)\right)}
	\end{align}   
	under the condition 	$$
	\frac{2}{q'}+\frac{n}{p'}=2,\quad \frac{n+1}{2}<p'\leq \infty.
	$$

	The inequality  (\ref{duali})  can be use  to prove an inhomogeneous  Strichartz inequality. Consider the  inhomogeneous system
	\begin{align}\label{Inhomogeneous}
		\left\{\begin{array}{l}
			i \partial_t {\gamma}(t)=[L, \gamma(t)]+i R(t), \\
			\gamma |_{t=0}=0,
		\end{array}\right.
	\end{align}
	where $R(t)$ is a self-adjoint operator on $L^2\left(X\right)$ and   bounded for almost every $t$. Then by   Duhamel’s principle, the  solution  to the inhomogeneous system (\ref{Inhomogeneous}) can be written as
	$$
	\gamma(t)= \int_{0}^t e^{i(t-s) L} R(s) e^{i(s-t) L} d s .$$

	\begin{theorem}[Inhomogeneous Strichartz inequality]\label{Innn}
		Assume $L$ is a non-negative self-adjoint operator on $L^2(X)$ satisfying Assumption $(\mathcal{A})$ and $(\mathcal{B})$. Assume that $p, q, n \geq 1$ satisfy
		$$
		1\leq p <\frac{n+1}{n-1}\text { and } \frac{2}{q}+\frac{n}{p}=n.
		$$
		Then  the solution to the  inhomogeneous system (\ref{Inhomogeneous})
		satisfies
		$$
		\left\|\rho_{\gamma(t)}\right\|_{L^{q}\left(\mathcal{I}, L^{p}\left(X\right)\right)} \leqslant C_S \left\|\int_{\mathcal{I}} e^{-i s L} R(s)  e^{i s L} d s\right\|_{\mathfrak{S}^{\frac{2 p}{p+1}}}
		$$
		for a constant $C_S>0$ independent of $R$.
	\end{theorem}
	\begin{proof}
		We prove this result using the duality argument.  
		Now  for a time-dependent potential $V(t, x) $, we have
		$$
		\begin{aligned}
			\left|  \int_{\mathcal{I}} \int_{X} V(t,  x) \rho_{\gamma(t)}(x) d x d t \right|&=	\left|  \int_{\mathcal{I}} \operatorname{Tr}(V(t, x) \gamma(t)) d t \right| \\
			& =\left| \int_{\mathcal{I}}    \int_{0}^t  \operatorname{Tr}\left(  V(t, x) e^{i t L} e^{-i s L} R(s) e^{i s L} e^{-i t L}\right) d s d t\right| \\
			& =\left| \int_{\mathcal{I}}    \int_{0}^t  \operatorname{Tr}\left(e^{-i t L} V(t, x) e^{i t L} e^{-i s L} R(s) e^{i s L}\right) d s d t\right| \\
			& \leq \operatorname{Tr}\left(\left(\int_{\mathcal{I}} e^{-i t L} V(t, x)  e^{i t L} d t\right)\left(\int_{\mathcal{I}}  e^{-i s L} R(s) e^{i s L} d s\right)\right),\\
			& \leq  \left\| \int_{\mathcal{I}} e^{-i t L} V(t, x) e^{i t L} d t\right\|_{\mathfrak{S}^{2p'}}  \left\| \int_{\mathcal{I}} e^{-i sL} R(s) e^{i s L} d s\right\|_{\mathfrak{S}^{\frac{2 p}{p+1}}},
		\end{aligned}
		$$
		where in the last line we used  Hölder's inequality for traces.  Now using the orthonormal Strichartz estimates for the homogeneous version (\ref{duali}),  we get  
		$$
		\begin{aligned}
			\left|  \int_{\mathcal{I}} \int_{X} V(t,  x) \rho_{\gamma(t)}(x) d x d t \right| 
			& \leq  C_S\left\|  V \right\|_{L^{q'}\left(\mathcal{I}, L^{p'}\left(X\right)\right)}   \left\| \int_{\mathcal{I}} e^{-i sL} R(s) e^{i s L} d s\right\|_{\mathfrak{S}^{\frac{2 p}{p+1}}}.
		\end{aligned}
		$$  Finlly using the duality argument, we get 
		$$	\left\|\rho_{\gamma(t)}\right\|_{L^{q}\left(\mathcal{I}, L^{p}\left(X\right)\right)} \leq   C_S \left\| \int_{\mathcal{I}} e^{-i sL} R(s)  e^{i s L} d s\right\|_{\mathfrak{S}^{\frac{2 p}{p+1}}}.$$
	\end{proof}

	Now consider the general inhomogeneous system
	\begin{align}\label{inhomogeneous}
		\left\{\begin{array}{l}
			i \partial_t {\gamma}(t)=[L, \gamma(t)]+i R(t), \\
			\gamma |_{t=0}=\gamma_0,
		\end{array}\right.
	\end{align}
	where $R(t)$ is a self-adjoint operator on $L^2\left(X\right)$ which is bounded for almost every $t$. Then by   Duhamel’s principle, the  solution  to the inhomogeneous system (\ref{inhomogeneous}) can be written as
	$$
	\gamma(t)=e^{i t L} \gamma_0 e^{-i t L}+\int_{0}^t e^{i(t-s) L} R(s) e^{i(s-t) L} d s .$$
	Thus using (\ref{7}) and Theorem \ref{Innn}, we immediately  can conclude that 	\begin{theorem}\label{In}
		Assume $L$ is a non-negative self-adjoint operator on $L^2(X)$ satisfying Assumption $(\mathcal{A})$ and $(\mathcal{B})$. Assume that $p, q, n \geq 1$ satisfy
		$$
		1\leq p <\frac{n+1}{n-1} \text { and } \frac{2}{q}+\frac{n}{p}=n
		$$
		and let $\gamma_0\in {\mathfrak{S}^{\frac{2p}{p+1}}}$. Then  the solution to the  inhomogeneous system (\ref{inhomogeneous})
		satisfies
		$$
		\left\|\rho_{\gamma(t)}\right\|_{L^{q}\left(\mathcal{I}, L^{p}\left(X\right)\right)} \leqslant C_S\left( \|\gamma_0\|_{\mathfrak{S}^{\frac{2 p}{p+1}}}+ \left\|\int_{\mathcal{I}} e^{-i s L} R(s)  e^{i s L} d s\right\|_{\mathfrak{S}^{\frac{2 p}{p+1}}}\right)
		$$
		for a constant $C_S>0$ independent of $\gamma_0$ and $R$.
	\end{theorem}

	\section{Examples}\label{sec5}
	In this section, we will apply the obtained results to study Strichartz estimates and their orthonormal version  
	for Laplacian, Hermite, twisted Laplacians, Laguerre operators, and $(k, a)$-generalized operators. It is worth noticing that the list of applications is not exhaustive, since we just intend to show the generality of our theory. Apart from these applications, one can find more applications in other setting such as Schrödinger operators with smooth potentials. 
	
	\subsection{Laplacian operator:}
	Let $L=-\Delta$ be the classical Laplacian operator on $X=\mathbb{R}^d$ with $d\geq 1$.  Let $K_{it}(x, y)$ denote the kernel of the  Schr\"odinger semigroup $e^{-it \Delta}$. Then it is well-known that  (see for example \cite{RS}) 
	$$ \sup_{x,y\in \mathbb{R}^d}|K_{it}(x,y)|\lesssim |t|^{-\frac{d}{2}}$$
	with $T_0=\R.$  
	Therefore  $K_{it}(x, y)$  satisfies Assumption $(\mathcal{B})$  with $n=d$. In this case, Theorem \ref{OSI} reduces to Strichartz estimates associated with the classical Laplacian operator for a system of orthonormal functions as follows:
	\begin{theorem}\label{D}
		Let  $p,q, d\geq 1$ satisfy \begin{equation*}
			\frac{2}{q}+\frac{d}{p}=d. 
		\end{equation*}
		(i) If  $  1\leq p<\frac{d+1}{d-1},$  then 
		\begin{equation}\label{Equ111111}
			\left\|\sum_{j\in J}n_j|e^{-it\Delta}f_j|^2\right\|_{L^q(\R,L^p(\R^n))} \leq C\left\|\{n_j\}_{j\in J}\right\|_{\ell^\frac{2p}{p+1}},
		\end{equation}
		holds for any (possibly infinite) orthonormal system $\{f_j\}_{j\in J}$ in $L^2(\R^d)$ and any sequence $\{n_j\}_{j\in J}$ in $\mathbb{C}$,
		where $C>0$ is independent of $\{f_j\}_{j\in J}$ and $\{n_j\}_{j\in J}$.\\
		
		(ii) If 
		\begin{align*}
			\begin{cases}  d=1, \\
				p=\infty,
			\end{cases}or ~~ \begin{cases}  d=2  \\
				3\leq p<\infty,
			\end{cases} or ~~ \begin{cases}  d\geq3\\
				\frac{d+1}{d-1}\leq p\leq \frac{d}{d-2},
			\end{cases}
		\end{align*}   then for  any $1\leq\lambda<q$
		\begin{equation}\label{eq22222}
			\left\|\sum_{j\in J}n_j|e^{-it\Delta}f_j|^2\right\|_{L^q(\R,L^p(\R^d))} \leq C\left\|\{n_j\}_{j\in J}\right\|_{\ell^\lambda},
		\end{equation}
		holds for any (possibly infinite) orthonormal system $\{f_j\}_{j\in J}$ in $L^2(\R^d)$ and any sequence $\{n_j\}_{j\in J}$ in $\mathbb{C}$,
		where $C>0$ is independent of $\{f_j\}_{j\in J}$ and $\{n_j\}_{j\in J}$.
	\end{theorem}
	Note that the above orthonormal  Strichartz inequality for the Laplacian  can be found in   \cite{frank,BHLNS, FS}. Furthermore,  in \cite{BHLNS} the authors also proved that  \eqref{Equ111111} does not hold for other exponent $\lambda>\frac{2p}{p+1}$. Moreover, when $p=\frac{d+1}{d-1}$, the estimate \eqref{Equ111111} holds for all exponent $\lambda<\frac{2p}{p+1}$ and fails for $\lambda=\frac{2p}{p+1}$.

	\begin{remark}
		Let $L=-\Delta$ be the classical Laplacian operator on $X=\mathbb{R}^d$ with $d\geq 2$.  Let $K_{it}(x, y)$ denote the kernel of the truncated wave semigroup $e^{it \sqrt{-\Delta}}\psi(\sqrt{-\Delta})$, where $\psi\in C_c^\infty([\frac{1}{2},2])$. Then it is well-known that  (see for example \cite{GV, GPW2008})
		$$ \sup_{x,y\in \mathbb{R}^d}|K_{it}(x,y)|\lesssim_\psi |t|^{-\frac{d-1}{2}}$$
		with $T_0=\R.$  
		Therefore  $K_{it}(x, y)$  satisfies Assumption $(\mathcal{B'})$  with $n=d-1$. In this case, Theorem \ref{truncated} reduces to the frequency localized Strichartz estimates associated with the classical wave operator for a system of orthonormal functions as follows:
		\begin{theorem}\label{wave}
			Suppose $p,q\geq 1$ and $d\geq2$ satisfy
			\begin{equation*}
				\frac{2}{q}+\frac{d-1}{p}=d-1.
			\end{equation*}
			
			(i) If $1\leq p<\frac{d}{d-2}$, then 
			\begin{equation*}
				\left\|\sum_{j\in J}n_j|e^{it \sqrt{-\Delta}}\psi(\sqrt{-\Delta})f_j|^2\right\|_{L^q(\mathbb{R},L^p(\mathbb{R}^d))} \lesssim_\psi \left\|\{n_j\}_{j\in J}\right\|_{\ell^\frac{2p}{p+1}},
			\end{equation*}
			holds for any (possibly infinite) orthonormal system $\{f_j\}_{j\in J}$ in $L^2(\mathbb{R}^d)$ and any sequence $\{n_j\}_{j\in J}$ in $\mathbb{C}$.\\
			
			(ii) If  \begin{align*}
				\begin{cases}  d=2, \\
					p=\infty,
				\end{cases}or ~~ \begin{cases}  d=3 \\
					3\leq p<\infty,
				\end{cases} or ~~ \begin{cases}  d\geq4\\
					\frac{d}{d-2}\leq p\leq \frac{d-1}{d-3},
				\end{cases}
			\end{align*} 
			then for $1\leq\lambda<q$
			\begin{equation*}
				\left\|\sum_{j\in J}n_j|e^{it \sqrt{-\Delta}}\psi(\sqrt{-\Delta}) f_j|^2\right\|_{L^q(\mathbb{R},L^p(\mathbb{R}^d))}\lesssim_\psi \left\|\{n_j\}_{j\in J}\right\|_{\ell^\lambda},
			\end{equation*}
			holds for any (possibly infinite) orthonormal system $\{f_j\}_{j\in J}$ in $L^2(\mathbb{R}^d)$ and any sequence $\{n_j\}_{j\in J}$ in $\mathbb{C}$.
		\end{theorem}
	\end{remark}
	\subsection{Hermite operator:}
	Let $L=H=-\Delta+|x|^2$ be the Hermite operator on $X=\mathbb{R}^d$ with $d\geq 1$.  Let $K_{it}(x, y)$ denote the kernel of the Hermite semigroup $e^{-it H}$. Then it is well-known that  (see for example \cite{NR2005}) 
	$$ \sup_{x,y\in \mathbb{R}^d}|K_{it}(x,y)|\lesssim |t|^{-\frac{d}{2}},\,|t|<\pi.$$  
	Therefore  $K_{it}(x, y)$  satisfies Assumption $(\mathcal{B})$  with $n=d$. In this case, Theorem \ref{OSI} reduces to Strichartz estimates associated with the Hermite operator for a system of orthonormal functions as follows:
	\begin{theorem}\label{H}		
		Let  $p,q$ satisfy \begin{equation*}
			\frac{2}{q}+\frac{d}{p}=d. 
		\end{equation*}
		(i) If  $  1\leq p<\frac{d+1}{d-1},$  then 
		\begin{equation*}
			\left\|\sum_{j\in J}n_j|e^{itH}f_j|^2\right\|_{L^q((-\pi, \pi),L^p(\R^n))} \leq C\left\|\{n_j\}_{j\in J}\right\|_{\ell^\frac{2p}{p+1}},
		\end{equation*}
		holds for any (possibly infinite) orthonormal system $\{f_j\}_{j\in J}$ in $L^2(\R^d)$ and any sequence $\{n_j\}_{j\in J}$ in $\mathbb{C}$,
		where $C>0$ is independent of $\{f_j\}_{j\in J}$ and $\{n_j\}_{j\in J}$.\\
		
		(ii) If 
		\begin{align*}
			\begin{cases}  d=1, \\
				p=\infty,
			\end{cases}or ~~ \begin{cases}  d=2  \\
				3\leq p<\infty,
			\end{cases} or ~~ \begin{cases}  d\geq3\\
				\frac{d+1}{d-1}\leq p\leq \frac{d}{d-2},
			\end{cases}
		\end{align*}   then for  any $1\leq\lambda<q$
		\begin{equation*}
			\left\|\sum_{j\in J}n_j|e^{itH}f_j|^2\right\|_{L^q((-\pi, \pi),L^p(\R^d))} \leq C\left\|\{n_j\}_{j\in J}\right\|_{\ell^\lambda},
		\end{equation*}
		holds for any (possibly infinite) orthonormal system $\{f_j\}_{j\in J}$ in $L^2(\R^d)$ and any sequence $\{n_j\}_{j\in J}$ in $\mathbb{C}$,
		where $C>0$ is independent of $\{f_j\}_{j\in J}$ and $\{n_j\}_{j\in J}$.
	\end{theorem}
	Note that part (i) of the above orthonormal  Strichartz inequality for the Hermite operator can be found in Mondal-Swain \cite{MS} and Bez-Hong-Lee-Nakamura-Sawano \cite{BHLNS} with the Schatten exponent $\frac{2p}{p+1}$ is optimal, however, result obtained in the part (ii) is not considered yet in the literature. 
	
	\subsection{Twisted Laplacians:}
	Let $L=\mathcal{L}=\frac{1}{2} \sum_{j=1}^{d}\left(Z_{j} \bar{Z}_{j}+\bar{Z}_{j} Z_{j}\right),
	$ be the twisted Laplacian  operator on $X=\mathbb{C}^d$, where $$Z_{j}=\frac{\partial}{\partial \zeta_{j}}+\frac{1}{2} \bar{\zeta}_{j}, ~~\bar{Z}_{j}=-\frac{\partial}{\partial \bar{\zeta}_{j}}+\frac{1}{2} \zeta_{j},\quad  j=1,2, \ldots d.$$  
	
	Then   $e^{-it \mathcal{L}}$  is a twisted convolution operator with 
	kernel  $K_{it}(\zeta), \zeta\in \mathbb{C}^d$, that satisfies   (see for example \cite{R2008})  
	$$ \sup_{\zeta \in  \mathbb{C}^d}|K_{it}(\zeta)|\lesssim |t|^{-d},\,|t|<\pi.$$  
	This shows that  $K_{it}(\zeta)$  satisfies Assumption $(\mathcal{B})$  with $n=2d$. In this case, Theorem \ref{OSI} reduces to Strichartz estimates for a system of orthonormal functions as follows:
	\begin{theorem}
		Let  $p,q$ satisfy \begin{equation*}
			\frac{1}{q}+\frac{d}{p}=d. 
		\end{equation*}
		(i) If  $ 1\leq p < \frac{2d+1}{2d-1},$  then 
		\begin{equation*}
			\left\|\sum_{j\in J} n_{j}\left|e^{i t \mathcal{L} } f_{j}\right|^{2}\right\|_{L^q((-\pi, \pi),L^p(\mathbb{C}^{d}))} \leq C\left\|\{n_j\}_{j\in J}\right\|_{\ell^\frac{2p}{p+1}},
		\end{equation*}
		holds for any (possibly infinite) orthonormal system $\{f_j\}_{j\in J}$ in $L^2(\C^d)$ and any sequence $\{n_j\}_{j\in J}$ in $\mathbb{C}$,
		where $C>0$ is independent of $\{f_j\}_{j\in J}$ and $\{n_j\}_{j\in J}$.\\
		
		(ii) If 
		\begin{align*}
			\begin{cases}  d=1  \\
				3\leq p<\infty,
			\end{cases} or ~~ \begin{cases}  d\geq2\\
				\frac{2d+1}{2d-1}\leq p\leq \frac{d}{d-1},
			\end{cases}
		\end{align*}   then for  any $1\leq\lambda<q$
		\begin{equation*}
			\left\|\sum_{j\in J} n_{j}\left|e^{i t \mathcal{L} } f_{j}\right|^{2}\right\|_{L^q((-\pi, \pi),L^p(\mathbb{C}^{d}))}\leq C\left\|\{n_j\}_{j\in J}\right\|_{\ell^\lambda},
		\end{equation*}
		holds for any (possibly infinite) orthonormal system $\{f_j\}_{j\in J}$ in $L^2(\C^d)$ and any sequence $\{n_j\}_{j\in J}$ in $\mathbb{C}$,
		where $C>0$ is independent of $\{f_j\}_{j\in J}$ and $\{n_j\}_{j\in J}$.	
	\end{theorem}
	Note that, part (i) of the above orthonormal  Strichartz inequality for the twisted Laplacian operator has been proved by  Ghosh-Mondal-Swain in \cite{GMS}. They   also proved that the Schatten exponent $\frac{2p}{p+1}$ is optimal.

	\subsection{Lagurre operator:}
	For $\alpha=(\alpha_1, \cdots \alpha_d)\in (-\frac{1}{2}, \infty)^d$, let $$L=L_{\alpha}=-\Delta+\sum_{j=1}^d\left(\frac{2\alpha_j+1}{x_j}\frac{\partial}{\partial x_j}\right) $$ be the Lagurre operator on $X=\mathbb{R}^d$ with the measure $d\mu(x)=dw_\alpha(x)=x^{2\alpha+1}dx$.   Also let  $K_{it}(x, y)$ denote the kernel of the Lagurre operator $e^{-it L_{\alpha}}$. Then it is well-known that  from  \cite{GM} 
	$$ \sup_{x,y\in \mathbb{R}^d}|K_{it}(x,y)|\lesssim |t|^{-\left(d+\sum_{j=1}^d \alpha_j\right)},\,|t|\leq \frac{\pi}{2}.$$  
	Therefore  $K_{it}(x, y)$  satisfies Assumption $(\mathcal{B})$  with $n=2\left(d+\sum_{j=1}^d \alpha_j\right)$. In this case, Theorem \ref{OSI} reduces to  Strichartz estimates associated with the  Laguerre
	operator for the system of orthonormal functions as follows:
	\begin{theorem}
		Let $p, q, d \geq 1$ satisfy
		$$
		\frac{1}{q}+\frac{d+\sum_{j=1}^d \alpha_j}{p}=d+\sum_{j=1}^d \alpha_j.
		$$
		
		(i) If  $	1 \leq p<\frac{2\left(d+\sum_{j=1}^d \alpha_j\right)+1}{2\left(n+\sum_{j=1}^d \alpha_j\right)-1} ,$  then 
		\begin{equation*}
			\left\|\sum_{j \in J} n_j\left|e^{i t L_\alpha} f_j\right|^2\right\|_{L^q\left(\left[-\frac{\pi}{2}, \frac{\pi}{2}\right], L^p\left(\mathbb{R}_{+}^d, dw_\alpha\right)\right)}  \leq C\left\|\{n_j\}_{j\in J}\right\|_{\ell^\frac{2p}{p+1}},
		\end{equation*}
		holds for any (possibly infinite) orthonormal system $\{f_j\}_{j\in J}$ in $L^2\left(\mathbb{R}_{+}^d, dw_\alpha\right)$ and any sequence $\{n_j\}_{j\in J}$ in $\mathbb{C}$,
		where $C>0$ is independent of $\{f_j\}_{j\in J}$ and $\{n_j\}_{j\in J}$.\\
		
		(ii) If 
		\begin{align*}
			&	\begin{cases}  1\leq 2\left(d+\sum_{j=1}^d \alpha_j\right)<2, \\
				\frac{2\left(d+\sum_{j=1}^d \alpha_j\right)+1}{2\left(d+\sum_{j=1}^d \alpha_j\right)-1}\leq p\leq \infty,
			\end{cases}or ~~ \begin{cases}   d+\sum_{j=1}^d \alpha_j=1  \\
				3\leq p<\infty,
			\end{cases} \\& or ~~ \begin{cases}   d+\sum_{j=1}^d \alpha_j >1\\
				\frac{2\left(d+\sum_{j=1}^d \alpha_j\right)+1}{2\left(d+\sum_{j=1}^d \alpha_j\right)-1}\leq p\leq \frac{d+\sum_{j=1}^d \alpha_j}{\left(d+\sum_{j=1}^d \alpha_j\right)-1},
			\end{cases}
		\end{align*}   then for  any $1\leq\lambda<q$
		\begin{equation*}
			\left\|\sum_{j \in J} n_j\left|e^{i t L_\alpha} f_j\right|^2\right\|_{L^q\left(\left[-\frac{\pi}{2}, \frac{\pi}{2}\right], L^p\left(\mathbb{R}_{+}^d, dw_\alpha\right)\right)}   \leq C\left\|\{n_j\}_{j\in J}\right\|_{\ell^\lambda},
		\end{equation*}
		holds for any (possibly infinite) orthonormal system $\{f_j\}_{j\in J}$ in $L^2\left(\mathbb{R}_{+}^d, dw_\alpha\right)$ and any sequence $\{n_j\}_{j\in J}$ in $\mathbb{C}$,
		where $C>0$ is independent of $\{f_j\}_{j\in J}$ and $\{n_j\}_{j\in J}$.
	\end{theorem}
	We refer to Feng-Song   \cite{GM} for a detailed study on the Laguerre operator and the proof of part (i) of the above orthonormal  Strichartz inequality associated with the Laguerre operator.

	\subsection{$ (k, a)$-generalized Laguerre operator:}
	Let   $a>0$ to be a deformation parameter and  $k$ be a multiplicity function. Then  consider  the $ (k, a)$-generalized Laguerre operator
	$$
	L=\Delta_{k, a}:=\frac{1}{a}\left(\|x\|^{a}-\|x\|^{2-a} \Delta_{k}\right),
	$$
	where $\Delta_k$ is the Dunkl Laplacian operator  on $X=\mathbb{R}^d.$ For a notations and detailed study on the Dunkl theory, we refer to \cite{MM} and references therein.
	
	Let $K_{it}(x, y)$ denote the kernel of the $ (k, a)$-generalized Laguerre   semigroup $e^{-it \Delta_{k, a}}$. Then from \cite{}, it is well-known that 
	$$ \sup_{x,y\in \mathbb{R}^d}|K_{it}(x,y)|\lesssim |t|^{-\frac{2\gamma+d+a-2}{a}},\,|t|<\frac{\pi}{2},$$   
	This shows that   $K_{it}(x, y)$  satisfies Assumption $(\mathcal{B})$  with $n=\frac{2(2\gamma+d+a-2)}{a}$. In this case, Theorem \ref{OSI} reduces to Strichartz estimates associated with $ (k, a)$-generalized Laguerre  operator for a system of orthonormal functions as follows:
	\begin{theorem}  Suppose $a=1,2$ and $k$ is a non-negative multiplicity function such that
		$
		a+2 \gamma+d-2>0.
		$ If $p, q, d \geqslant$ 1 such that
		$$
		\frac{1}{q}+\frac{2 \gamma+d+a-2}{pa}=\frac{2 \gamma+d+a-2}{a}.
		$$
		
		(i) If  $  	1 \leqslant p<\frac{4\gamma+2d+3a-4}{  4\gamma+2d+a-4},$  then 
		\begin{equation*}
			\left\|\sum_{  j \in J} n_{j }\left|e^{i t \Delta_{k, a}}f_{j }\right|^{2}\right\|_{L^q ((-\frac{\pi}{2},\frac{\pi}{2}), L_{k,a}^p(\mathbb{R}^d))}  \leq C\left\|\{n_j\}_{j\in J}\right\|_{\ell^\frac{2p}{p+1}},
		\end{equation*}
		holds for any (possibly infinite) orthonormal system $\{f_j\}_{j\in J}$ in $L_{k,a}^{2}\left(\mathbb{R}^d\right)$ and any sequence $\{n_j\}_{j\in J}$ in $\mathbb{C}$,
		where $C>0$ is independent of $\{f_j\}_{j\in J}$ and $\{n_j\}_{j\in J}$.\\
		
		(ii) If 
		\begin{align*}
			\begin{cases}  a\leq {2(2\gamma+d+a-2)}<2a, \\
				\frac{4\gamma+2d+3a-4}{  4\gamma+2d+a-4}\leq p\leq \infty,
			\end{cases}or ~~ \begin{cases} 2\gamma+d=2  \\
				3\leq p<\infty,
			\end{cases} or ~~ \begin{cases}  2\gamma+d>2\\
				\frac{4\gamma+2d+3a-4}{  4\gamma+2d+a-4}\leq p\leq \frac{ {2\gamma+d+a-2}}{{2\gamma+d-2}},
			\end{cases}
		\end{align*}   then for  any $1\leq\lambda<q$
		\begin{equation*}
			\left\|\sum_{  j \in J} n_{j }\left|e^{i t \Delta_{k, a}}f_{j }\right|^{2}\right\|_{L^q ((-\frac{\pi}{2},\frac{\pi}{2}), L_{k,a}^p(\mathbb{R}^d))}  \leq C\left\|\{n_j\}_{j\in J}\right\|_{\ell^\lambda},
		\end{equation*}
		holds for any (possibly infinite) orthonormal system $\{f_j\}_{j\in J}$ in $L_{k,a}^{2}\left(\mathbb{R}^d\right)$ and any sequence $\{n_j\}_{j\in J}$ in $\mathbb{C}$,
		where $C>0$ is independent of $\{f_j\}_{j\in J}$ and $\{n_j\}_{j\in J}$.
	\end{theorem}
	Note that part (i) of the above orthonormal  Strichartz inequality for  $ (k, a)$-generalized Laguerre operator has been proved by   Mondal-Song in \cite{MM}. 	Using the classical Strichartz estimates for the free Schr\"odinger propagator $e^{-i t \Delta_{k, a}} $  for orthonormal systems of initial data and the kernel relation between the semigroups $e^{-i t \Delta_{k, a}}$ and $e^{i \frac{t}{a}\|x\|^{2-a} \Delta_{k}},$ they also prove  Strichartz estimates for orthonormal systems of initial data associated with the Dunkl Laplacian operator $ \Delta_k $ on $\mathbb{R}^n$.

	\section{Application: Hartree equation}\label{sec6}
	The following global well-posedness for the Hartree equation in Schatten spaces can be obtained as an application of the orthonormal Strichartz inequalities.

	Consider $M$ couple of nonlinear  Hartree equations which describes the dynamics of $M$ fermions interacting via a  potential $w$  as follows
	$$
	\left\{\begin{array}{rlrl}
		i \partial_t u_1 & =\left(L+w * \rho\right) u_1, & & \left.u_1\right|_{t=0}=f_1 \\
		& \vdots \\
		i \partial_t u_M & =\left(L+w * \rho\right) u_M, & & \left.u_M\right|_{t=0}=f_M,
	\end{array}\right.
	$$
	where $(x, t) \in  X \times \mathbb{R},\left\{f_j\right\}_{j=1}^M$ is an orthonormal system in $L^2\left(X\right)$ and $\rho$ is a density function defined by $\rho(x, t)=\sum_{j=1}^M\left|u_j(x, t)\right|^2$.  One of the natural questions is to investigate the behaviour of solution of the system as $M \rightarrow \infty$.  In this case, we naturally arrive at the following operator-valued equivalent formulation 
	$$
	\left\{\begin{array}{l}
		i \partial_t \gamma=\left[L+w * \rho_\gamma, \gamma\right], \\
		\left.\gamma\right|_{t=0}=\gamma_0.
	\end{array}\right.
	$$
	Here $\gamma_0$ and $\gamma=\gamma(t)$ are bounded and self-adjoint operators on $L^2\left(X\right)$.

	Now using the inhomogeneous Strichartz inequality, we have the following well-posedness result for the Hartree equation.
	\begin{theorem}
		Assume $L$ is a non-negative self-adjoint operator on $L^2(X)$ that satisfies Assumption $(\mathcal{A})$ and $(\mathcal{B})$. Let $w\in L^{p^\prime}(X)$ and $p,q\geq 1$ satisfy
		\begin{equation*}
			1\leq p<\frac{n+1}{n-1}\;\text{and}\; \frac{2}{q}+\frac{n}{p}=n
		\end{equation*}

		\begin{enumerate}
			\item  (Local well-posedness)    For any $\gamma_0  \in  {\mathfrak{S}^{\frac{2 p}{p+1}}}$ with $R=\|\gamma_0 \|_{ {\mathfrak{S}^{\frac{2 p}{p+1}}}}<\infty$, there exists $T=T(R, \|w\|_{L^{p^\prime}(X)})$ (without loss of generality, we can assume $T< T_0$)  and a unique 
			$\gamma \in  C^0([0, T], {\mathfrak{S}^{\frac{2 p}{p+1}}} )$  
			satisfying
			$\rho_{\gamma} \in  L^q   \left([0, T], L^p(X)\right)$
			and
			
			\begin{align}\label{eq22}
				\begin{cases}
					i\partial_t \gamma = [L+w*\rho_\gamma, \gamma],\\
					\gamma|_{t=0 }= \gamma_0.
				\end{cases} 
			\end{align}	 
			\item  (Almost global well-posedness) For each $T>0$ and for  any $\gamma_0  \in  {\mathfrak{S}^{\frac{2 p}{p+1}}}$ with $\|\gamma_0 \|_{ {\mathfrak{S}^{\frac{2 p}{p+1}}}}\leq R_T$,  where  $R_T=R_T( \|w\|_{L^{p^\prime}(X)})$, then there exists a solution 	$\gamma \in  C^0([0, T], {\mathfrak{S}^{\frac{2 p}{p+1}}} )$  satisfying (\ref{eq22})   and   $\rho_{\gamma} \in  L^q   \left([0, T], L^p(X)\right)$.
		\end{enumerate}   
	\end{theorem}
	\begin{proof}
		(1)	 Let $R=\left\|\gamma_0\right\|_{  {\mathfrak{S}^{\frac{2 p}{p+1}}}}<\infty$ and $T=T(R)>0$ to be chosen later. We define the evolution space 
		$$
		\begin{aligned}
			X_T:=\left\{(\gamma, \rho) \in C^0\left([0, T], {\mathfrak{S}^{\frac{2 p}{p+1}}}\right) \times L^q\left([0, T], L^p\left(X\right)\right): \|(\gamma, \rho) \|_{X_T}<   {\max\{1, C_S \} } R\right\},
		\end{aligned}
		$$
		where $$
		\|(\gamma, \rho) \|_{X_T} =\|\gamma\|_{ C^0\left([0, T], {\mathfrak{S}^{\frac{2 p}{p+1}}} \right)} +\|\rho\|_{ L^q\left([0, T], L^p\left(X\right) \right)} . $$ 
		Let us define the  map $\Psi$ as
		$$
		\Psi(\gamma, \rho)(t)=e^{i t L} \gamma_0 e^{-i t L}-i \int_0^t e^{i(t-s) L}[w * \rho(s), \gamma(s)] e^{i(s-t) L} d s .
		$$
		We also define a map
		$$
		\Phi(\gamma, \rho)=\left(\Psi(\gamma, \rho), \rho\left[\Psi(\gamma, \rho)\right]\right),
		$$
		where  we used the notation $\rho\left[\gamma\right]=\rho_\gamma$ and in this formulation, (\ref{eq22}) is equivalent to $(\gamma, \rho_\gamma)=\Phi(\gamma, \rho_\gamma)$. 
		
		First, we evalute  $ 
		\left\|\Psi(\gamma, \rho)\right\|_{C^0\left([0, T], {\mathfrak{S}^{\frac{2 p}{p+1}}} \right)} 
		$ for any $t\in [0, T]. $ Now 
		\begin{align} \label{integral}
			\left\|\Psi(\gamma, \rho)(t)\right\|_{ {\mathfrak{S}^{\frac{2 p}{p+1}}} } &\leq \|e^{i t L} \gamma_0 e^{-i t L}\|_{ {\mathfrak{S}^{\frac{2 p}{p+1}}} }+ \int_0^t \| e^{i(t-s) L}[w * \rho(s), \gamma(s)] e^{i(s-t) L} \|_{ {\mathfrak{S}^{\frac{2 p}{p+1}}} }d s 
		\end{align}
		For the  first term in the RHS of  (\ref{integral}),     since  $\left(f_j\right)_j$ is orthonormal in $L^2(X)$, then $\left(e^{i t L} f_j\right)_j$ is as well for each $t$  and thus 
		$$
		\left\|e^{i t L} \gamma_0 e^{-i t L}\right\|_{{\mathfrak{S}^{\frac{2 p}{p+1}}} }=\left\|\gamma_0\right\|_{\mathfrak{S}^{\frac{2 p}{p+1}}}=R .
		$$
		For the second term of (\ref{integral}), we use   Hölder's inequality for Schatten spaces  
		$$
		\begin{aligned}
			\| e^{i(t-s) L}[w * \rho(s), \gamma(s)] e^{i(s-t) L} \|_{ {\mathfrak{S}^{\frac{2 p}{p+1}}} }& \leq   \|w * \rho(s) \|_{L^\infty(X)} \| \gamma (s) \|_{ {\mathfrak{S}^{\frac{2 p}{p+1}}} }\\
			& \leq   \|w \|_{L^{p'}(X)}\| \rho(s) \|_{_{L^{p}(X)}} \| \gamma (s) \|_{ {\mathfrak{S}^{\frac{2 p}{p+1}}} }
		\end{aligned}
		$$
		Thus for any $(\gamma, \rho)\in X_T$,  from (\ref{integral}), we get 
		\begin{align}\label{eq11}\nonumber
			\left\|\Psi(\gamma, \rho)\right\|_{_{C^0\left([0, T], {\mathfrak{S}^{\frac{2 p}{p+1}}} \right)} } &\leq R+ \int_0^T    \|w \|_{L^{p'}(X)}\| \rho(s) \|_{_{L^{p}(X)}} \| \gamma (s) \|_{ {\mathfrak{S}^{\frac{2 p}{p+1}}} } d s \\\nonumber
			&\leq R+ T^{\frac{1}{q'}} \|w \|_{L^{p'}(X)}   \| \rho \|_{_{L^q\left([0, T], L^{p}(X)\right)}} \| \gamma \|_{_{C^0\left([0, T], {\mathfrak{S}^{\frac{2 p}{p+1}}} \right)} }\\
			&\leq R+ T^{\frac{1}{q'}}   {\max\{1, C_S^2 \} } R^2 \|w \|_{L^{p'}(X)} .
		\end{align}
		Now our  aim is to estimate    $ 
		\left\|\rho[ \Psi(\gamma, \rho)]\right\|_{_{ L^q\left([0, T], L^p\left(X\right) \right)}  } 
		$. From the inhomogeneous Strichartz estimate in Theorem \ref{In}, we get 
		\begin{align}\label{eq12}\nonumber
			&\left\|\rho[ \Psi(\gamma, \rho)]\right\|_{_{ L^q\left([0, T], L^p\left(X\right) \right)}  }\\\nonumber
			&\leq 	 C_S\|\gamma_0\|_{\mathfrak{S}^{\frac{2 p}{p+1}}}+ \left\|\int_{\mathbb{R}} e^{-i s L}e^{i t L}|[w * \rho(s), \gamma(s)] | e^{-i t L}e^{i s L} d s\right\|_{\mathfrak{S}^{\frac{2 p}{p+1}}}\\
			&\leq 	 C_S R+ C_S T^{\frac{1}{q'}}   {\max\{1, C_S^2 \} } R^2 \|w \|_{L^{p'}(X)} ,
		\end{align}
		where for the first term of the above estimate,  we used $\left\|\gamma_0\right\|_{\mathfrak{S}^{\frac{2 p}{p+1}}}=R$ and for the second term, we employed the same argument as we used in  (\ref{eq11}). Thus from (\ref{eq11}) and \eqref{eq12}, we get 
		
		\begin{align}
			\|\Psi(\gamma, \rho)\|_{X_T}\leq (1+C_S)\left[ 	R+ T^{\frac{1}{q'}}   {\max\{1, C_S^2 \} }R^2 \|w \|_{L^{p'}(X)} \right].
		\end{align}
		Now, choose  $T=T(R)>0$ small enough so that  \begin{align*}
			\|\Psi(\gamma, \rho)\|_{X_T}\leq  \max\{1, C_S \} R. 
		\end{align*}
		This implies that 	$\Phi(\gamma, \rho)\in X_T$  for $(\gamma, \rho)\in X_T$, i.e., $\Phi $ maps $X_T$ to itself. 
		We can give a similar argument to show that $\Phi$ is a contraction on $X_T$.  Since $X_T$ is a Banach space and  $\Phi $ is a contraction map on  $X_T$, by Banach's fixed point theory,  $\Phi$ has a unique fixed point on $X_T$ which is a solution to the Hartree equation (\ref{eq22}) on $[0, T]$.\\
		
		(2) Now,  we first fix an arbitrary $T>0$.   Then estimates   (\ref{eq11}) and \eqref{eq12}    yield that
		\begin{align}
			\|\Psi(\gamma, \rho)\|_{X_T} \leq (1+C_S) \left[ 	 \|\gamma_0\|_{\mathfrak{S}^{\frac{2 p}{p+1}}}+ T^{\frac{1}{q'}}   \|w \|_{L^{p'}(X)}  \|(\gamma, \rho)\|_{X_T}^2 \right].
		\end{align}
		Keeping  this in mind, we choose $R_T=R_T\left( \|w \|_{L^{p'}(X)} \right)$ small enough  so that we can find $M>0$ such that for any $y \in[0, M]$, it holds
		$$
		(1+C_S) \left[ 	 \|\gamma_0\|_{\mathfrak{S}^{\frac{2 p}{p+1}}}+ T^{\frac{1}{q'}}   \|w \|_{L^{p'}(X)} y^2 \right] \leq M
		$$
		as long as $\|\gamma_0\|_{\mathfrak{S}^{\frac{2 p}{p+1}}} \leq R_T$. Therefore, if we define the space $X_{T, M}$ by
		$$
		X_{T, M}:=\left\{(\gamma, \rho) \in X_T:\|(\gamma, \rho)\|_{X_T} \leq M\right\},
		$$
		then wecan  see that $\Phi$ maps $ X_{T, M} $ to $ X_{T, M}$ itself. By choosing $R_T$ even further smaller, we also can show that $\Phi$ is a contraction map on $X_{T, M}$ in a similar way as we discussed earlier. Thus  by the Banach  fixed point theorem, we find a solution 	$\gamma \in  C^0([0, T], {\mathfrak{S}^{\frac{2 p}{p+1}}} )$  
		satisfying
		$\rho_{\gamma} \in  L^q   \left([0, T], L^p(X)\right)$  for   fix an arbitrary $T > 0$.
	\end{proof}

	\section{Generalized Strichartz results by Kenig-Ponce-Vega}\label{sec7}
	This section is devoted to extending the Strichartz (or ``global smoothing") estimates considered in \cite{Kenig-Ponce-Vega} involving a single function to a system of orthonormal functions using the techniques used in the proof Theorem \ref{OSI}.
	\subsection{Orthonormal Strichartz inequalities on $\mathbb{R}$}
	First, we recall the notion of the following class $\mathcal{A}$ of phase functions introduced by Kenig-Ponce-Vega \cite{Kenig-Ponce-Vega}.
	\begin{definition}\label{Definition}
		We say $\phi\in\mathcal{A}$ if
		\begin{itemize}
			\item $\phi: \Omega\rightarrow\mathbb{R}$ for some open $\Omega\in\mathbb{R}$, $\phi\in C^3(\Omega)$, and $\Omega$ is a finite union of intervals,
			\item the set $S_\phi=\{\xi\in \bar{\Omega}\cup \{\pm\infty\}: \phi''(\xi)=0 \text{ or }\lim\limits_{\bar{\xi}\rightarrow\xi} \phi''(\bar{\xi})=\phi''(\xi)=\pm\infty\}$ is finite,
			\item if $\xi_0\in S_\phi$ with $\xi_0\neq \pm\infty$, then there exist constants $\varepsilon,c_1,c_2$ and $\alpha\neq0$ such that for $|\xi-\xi_0|<\varepsilon$,
			\begin{equation*}
				c_1|\xi-\xi_0|^{\alpha-2}\leq |\phi''(\xi)|\leq c_2|\xi-\xi_0|^{\alpha-2},
			\end{equation*}
			\item if $\xi_0=\pm\infty\in S_\phi$, then there exist constants $\varepsilon,c_1,c_2$ and $\alpha\neq0$ such that for $|\xi|>\frac{1}{\varepsilon}$,
			\begin{equation*}
				c_1|\xi|^{\alpha-2}\leq |\phi''(\xi)|\leq c_2|\xi|^{\alpha-2},
			\end{equation*}
			\item $\phi''$ has a finite number of changes of monotonicity.
		\end{itemize}
	\end{definition}
	\begin{remark} It is easily checked that if $\phi$ is not linear and $\phi(\xi)=R(\xi)^\alpha$, where $R(\xi)$ is a rational function and $\alpha\in\mathbb{R}$, then $\phi\in\mathcal{A}$.
	\end{remark}
	Let $\phi\in\mathcal{A}$. Then for $\nu\geq 0$ and $(x,t)\in\mathbb{R}\times\mathbb{R},$, the operator $\mathcal{W}_\nu(t) $  is     defined  as
	\begin{equation*}
		\mathcal{W}_\nu(t) f(x)=\int_{\Omega} e^{i\left(t\phi(\xi)+x\xi\right)}|\phi''(\xi)|^\frac{\nu}{2} \hat{f}(\xi) d\xi.
	\end{equation*}
	In \cite{Kenig-Ponce-Vega}, Kenig-Ponce-Vega obtained the following Strichartz estimate for the operator $\mathcal{W}_\nu$ involving a single function.
	\begin{theorem}\label{One dim}
		Let $\phi\in\mathcal{A}$. Then for any $\nu\in [0,1]$ and $(p,q)=\left(\frac{2}{1-\nu},\frac{4}{\nu}\right)$,
		\begin{align*}
			\|\mathcal{W}_\frac{\nu}{2}(t) f\|_{L^q(\mathbb{R},L^p(\mathbb{R}))}&\leq C\|f\|_{L^2(\mathbb{R})},\\
			\left\|\int_{\mathbb{R}}\mathcal{W}_\nu(t-s)F(\cdot, s)ds\right\|_{L^q(\mathbb{R},L^p(\mathbb{R}))}&\leq C\|F\|_{L^{q'}(\mathbb{R},L^{p'}(\mathbb{R}))},
		\end{align*}
		for a constant $C>0$ independent of $f$ and $F$.
	\end{theorem}
	The above Strichartz estimates on $\R$ can be interpreted as restriction estimates on certain hypersurfaces of $\mathbb{R}^2$. In fact, let $S$ be the noncompact hypersurface of $\mathbb{R}^2$ defined by
	$$S=\{(\xi, \phi(\xi): \xi\in\mathbb{R}\},$$
	endowed with the measure
	$$d\sigma_\nu(\xi,\phi(\xi))=\frac{|\phi''(\xi)|^\frac{\nu}{2}}{\sqrt{1+|\phi'(\xi)|^2}}d\xi.$$
	One has the following identity. For all $F\in L^1(S,d\sigma_\nu)\cap L^2(S,d\sigma_\nu)$ and $(x,t)\in\mathbb{R}^2$, we have
	\begin{align*}
		E^\nu_SF(x,t)&=\frac{1}{2\pi}\int_S e^{i(x,t)\cdot (\xi,\phi(\xi))} F(\xi,\phi(\xi))d\sigma_\nu(\xi,\phi(\xi))\\
		&=\frac{1}{2\pi}\int_\mathbb{R}e^{i\left(t\phi(\xi)+x\xi\right)} F(\xi,\phi(\xi))|\phi''(\xi)|^\frac{\nu}{2} d\xi.
	\end{align*}
	In particular, choose $f:\mathbb{R}\rightarrow \mathbb{C}$ such that $\hat{f}(\xi)=|\phi''(\xi)|^\frac{\nu}{4}F(\xi,\phi(\xi))$. Then we see that
	$\|f\|_{L^2(\mathbb{R})}=\|F\|_{L^2(S,d\sigma_\nu)}$ and
	\begin{equation}\label{relation}
		E^\nu_SF(x,t)=\frac{1}{2\pi}\mathcal{W}_\frac{\nu}{2}(t) f(x).
	\end{equation}
	Since $\mathcal{W}_\frac{\nu}{2}(t)$ is bounded from $L^2(\mathbb{R})$ to $L^q(\mathbb{R}, L^p(\mathbb{R}))$, we have the following restriction estimate.
	\begin{theorem}Suppose $\phi\in\mathcal{A}$, $\nu\in [0,1]$ and $(p,q)=\left(\frac{2}{1-\nu},\frac{4}{\nu}\right)$.
		Let $S=\{(\xi, \phi(\xi): \xi\in\mathbb{R}\}$
		be the noncompact hypersurface of $\mathbb{R}^2$ endowed with the measure $d\sigma_\nu(\xi,\phi(\xi))=\frac{|\phi''(\xi)|^\frac{\nu}{2}}{\sqrt{1+|\phi'(\xi)|^2}}d\xi$. Then for any $F\in L^2(S,d\sigma_\nu)$,
		\begin{equation*}
			\|E^\nu_SF\|_{L^q(\mathbb{R}, L^p(\mathbb{R}))\lesssim \|F\|_{L^2(S,d\sigma_\nu)}.}
		\end{equation*}
	\end{theorem}

	For $\phi\in\mathcal{A}$ and  $Re(z)\geq 0$,  we define 
	\begin{equation*}
		I_z(x,t)=\int_{\Omega} e^{i\left(t\phi(\xi)+x\xi\right)}|\phi''(\xi)|^\frac{z}{2} d\xi, \quad (x,t)\in\mathbb{R}\times\mathbb{R}.
	\end{equation*}
	Based on \cite[Lemma 2.7]{Kenig-Ponce-Vega}, we can estimate decay rates as follows.
	\begin{lemma}\label{decay-1}
		Let $\phi\in\mathcal{A}$. Then for $Re(z)=1$,
		\begin{equation*}
			|I_z(x,t)|\leq C(1+|Im(z)|)|t|^{-\frac{1}{2}},
		\end{equation*}
		for a constant $C>0$ independent of $x$ and $t$.
	\end{lemma}
	Using the aforementioned decay estimate and considering the approach employed to derive  Theorem \ref{OSI}, we now extend Strichartz estimates for the operator $\mathcal{W}_\nu$ from a single function to a system of orthonormal initial data.  
	\begin{theorem}\label{real line}
		Let $\phi\in\mathcal{A}$, $\nu\in [\frac{2}{3},1)$ and $(p,q)=\left(\frac{1}{1-\nu},\frac{2}{\nu}\right)$. 
		For any (possibly infinite) orthonormal system $\{f_j\}_{j\in J}$ in $L^2(\mathbb{R})$ and any sequence $\{n_j\}_{j\in J}$ in $\mathbb{C}$, we have
		\begin{equation*}
			\left\|\sum_{j\in J}n_j|\mathcal{W}_\frac{\nu}{2}(t)f_j|^2\right\|_{L^q(\mathcal{\mathbb{R}},L^p(\mathbb{R}))} \leq C\left\|\{n_j\}_{j\in J}\right\|_{\ell^\frac{2p}{p+1}},
		\end{equation*}
		where $C>0$ is independent of $\{f_j\}_{j\in J}$ and $\{n_j\}_{j\in J}$.
	\end{theorem}
	\begin{proof}
		Thanks to the duality principle Theorem \ref{duality-principle}, the desired estimate holds if and only if
		\begin{equation}\label{1/2}	\left\|W_1\mathcal{T}_\nu W_2\right\|_{\mathfrak{S}^\frac{2}{\nu}(L^2(\mathbb{R}, L^2(\mathbb{R})))}\lesssim \left\|W_1\right\|_{L^\frac{4}{2-\nu}(\mathbb{R}, L^\frac{2}{\nu}(\mathbb{R}))}  \left\|W_2\right\|_{L^\frac{4}{2-\nu}(\mathbb{R}, L^\frac{2}{\nu}(\mathbb{R}))}  
		\end{equation}
		for all $W_1,W_2\in L^\frac{4}{2-\nu}(\mathbb{R}, L^\frac{2}{\nu}(\mathbb{R}))$, where $\mathcal{T}_\nu=\mathcal{W}_\frac{\nu}{2}(t)\left(\mathcal{W}_\frac{\nu}{2} (t) \right)^*$.
		
		Indeed, 
		\begin{align*}
			\mathcal{T}_\nu F(x,t)&=\int_{\mathbb{R}} {\W}_{\nu}(t-s)F(\cdot,s)ds\\
			&=\int_{\mathbb{R}}\int_{\mathbb{R}}I_\nu (x-y,t-s)F(y,s)dyds.
		\end{align*}
		For any $\varepsilon>0$, we define the operator $T_\varepsilon$  (similar as in  Lemma \ref{Schatten}) by  
		\begin{equation*}
			T_\varepsilon F(x,t)=\int_{\mathbb{R}}\int_{\mathbb{R}} \mathcal{K}_\varepsilon (x-y,t-s)F(y,s)\; dy ds,
		\end{equation*}
		where $\mathcal{K}_\varepsilon (x,t)=\chi_{|t|>\varepsilon}(t)I_\nu (x,t)$. 
		Once we prove
		\begin{equation}\label{varepsilon}	\left\|W_1T_\varepsilon W_2\right\|_{\mathfrak{S}^\frac{2}{\nu}(L^2(\mathbb{R}, L^2(\mathbb{R})))}\lesssim \left\|W_1\right\|_{L^\frac{4}{2-\nu}(\mathbb{R}, L^\frac{2}{\nu}(\mathbb{R}))}  \left\|W_2\right\|_{L^\frac{4}{2-\nu}(\mathbb{R}, L^\frac{2}{\nu}(\mathbb{R}))},
		\end{equation}
		for some $C>0$ independent of $\varepsilon$, then \eqref{1/2} follow by taking $\varepsilon\rightarrow 0$.
		
		In order to obtain \eqref{varepsilon},  we use Stein's analytic complex interpolation. We first construct an analytic family of operators for it. For any $\varepsilon>0$, we consider the analytic family of operators $T_{z,\varepsilon}$ defined by
		\begin{equation*}
			T_{z,\varepsilon}F(x,t)=\int_{\mathbb{R}}\int_{\mathbb{R}} \mathcal{K}_{z,\varepsilon} (x-y,t-s)F(y,s)\; dy ds,
		\end{equation*}
		where $z\in\mathbb{C}$ with $Re(z)\in [0,1]$ and
		\begin{equation*}
			\mathcal{K}_{z,\varepsilon}(x,t)=t^{-1+\frac{z}{\nu}}\chi_{|t|>\varepsilon}(t)I_z(x,t).
		\end{equation*}
		Clearly  $T_\varepsilon =T_{\nu,\varepsilon}$ and  it follows from Lemma \ref{decay-1} that for $Re(z)=1$
		\begin{equation}\label{Decay-1}
			\sup_{x\in \mathbb{R}}|\mathcal{K}_{z,\varepsilon}(x,t)|\leq C(1+|Im(z)|)|t|^{-\frac{3}{2}+\frac{Re(z)}{\nu}}.
		\end{equation}
		Once we prove the following two estimates
		\begin{equation}\label{T_z}
			\left\|W_1T_{z,\varepsilon} W_2\right\|_{\mathfrak{S}^2(L^2(\mathbb{R}, L^2(\mathbb{R})))}\leq C(Im(z)\|W_1\|_{L^{\frac{4\nu}{2-\nu}}\left(\mathbb{R},L^2(\mathbb{R})\right)}\|W_2\|_{L^{\frac{4\nu}{2-\nu}}\left(\mathbb{R},L^2(\mathbb{R})\right)}
		\end{equation}
		for $Re(z)=1$ and 
		\begin{equation}\label{Hilbert}
			\left\|W_1T_{z,\varepsilon} W_2\right\|_{\mathfrak{S}^\infty(L^2(\mathbb{R}, L^2(\mathbb{R})))}\leq C(Im(z))\left\|W_1\right\|_{L^\infty(\mathbb{R}, L^\infty(\mathbb{R}))}\left\|W_2\right\|_{L^\infty(\mathbb{R}, L^\infty(\mathbb{R}))}
		\end{equation}
		for $Re(z)=0$, where $C(Im(z))$ is a constant that grows exponentially with $Im(z)$, then, by Stein's complex interpolation method in Lemma \ref{Schatten-bound}, we have the required estimate  \eqref{varepsilon}.
		
		For $Re(z)=1$, our first aim is to prove   \eqref{T_z}. By \eqref{Decay-1}, the kernel of $W_1 T_{z,\varepsilon} W_2$ is bounded by 
		\begin{equation*}
			C(1+|Im(z)|)|W_1(x,t)||t-s|^{-\frac{3}{2}+\frac{1}{\nu}}|W_2(y,s)|,
		\end{equation*}
		and combined with the Hardy-Littlewood-Sobolev inequality, we have
		\begin{align*}
			\left\|W_1 T_{z,\varepsilon} W_2\right\|^2_{\mathfrak{S}^2(L^2(\mathbb{R}, L^2(\mathbb{R})))}&\leq C(1+|Im(z)|)^2\int_{\mathbb{R}^2}\int_{\mathbb{R}^2}\frac{|W_1(x,t)|^2|W_2(y,s)|^2}{|t-s|^{3-\frac{2}{\nu}}}dxdydtds\\
			&=C(1+|Im(z)|)^2\int_{\mathbb{R}^2}\frac{\|W_1(\cdot,t)\|_{L^2(\mathbb{R})}^2\|W_2(\cdot,s)\|_{L^2(\mathbb{R})}^2}{|t-s|^{3-\frac{2}{\nu}}}dtds\\
			&\leq C^2(Im(z)^2\|W_1\|^2_{L^{\frac{4\nu}{2-\nu}}\left(\mathbb{R},L^2(\mathbb{R})\right)}\|W_2\|^2_{L^{\frac{4\nu}{2-\nu}}\left(\mathbb{R},L^2(\mathbb{R})\right)}.
		\end{align*}
		On the other hand, for $Re(z)=0$, our next aim is to prove \eqref{Hilbert}.    Using the fact that $\mathfrak{S}^\infty$-norm is the operator norm and
		\begin{align*}
			\|W_1T_{z,\varepsilon}W_2\|_{\mathfrak{S}^\infty(L^2(\mathbb{R},L^2(\mathbb{R})))}&= \|W_1T_{z,\varepsilon}W_2\|_{L^2(\mathbb{R},L^2(\mathbb{R}))\rightarrow L^2(\mathbb{R},L^2(\mathbb{R}))}\\
			&\leq \|W_1\|_{L^\infty(\mathbb{R},L^\infty(\mathbb{R}))}\|T_{z,\varepsilon}\|_{L^2(\mathbb{R},L^2(\mathbb{R}))\rightarrow L^2(\mathbb{R},L^2(\mathbb{R}))}\|W_2\|_{L^\infty(\mathbb{R},L^\infty(\mathbb{R}))},
		\end{align*}
		to prove the estimate \eqref{Hilbert},  it's enough to prove that the operator $T_{z,\varepsilon}: L^2(\mathbb{R}, L^2(\mathbb{R}))\rightarrow L^2(\mathbb{R}, L^2(\mathbb{R}))$ is bounded with some constant depending only on $Im(z)$ exponentially when $Re(z)=0$. In fact, we rewrite
		\begin{align*}
			T_{z,\varepsilon} F(x,t)&=\int_\mathbb{R}\int_\mathbb{R}\mathcal{K}_{z,\varepsilon} (x-y,s)F(y,t-s)dyds\\
			&=\int_{|t|>\varepsilon} s^{-1+i\frac{Im(z)}{\nu}} ds \int_\mathbb{R} I_z(x-y,s)F(y,t-s)dy\\
			&=\int_{|t|>\varepsilon}  s^{-1+i\frac{Im(z)}{\nu}} \mathcal{W}_z(s)F(x,t-s)ds.
		\end{align*}
		The operator $\mathcal{W}_z(t)$ being unitary on $L^2(\mathbb{R})$ for $Re(z)=0$ for any $t\in\mathbb{R}$ and it reduces that
		\begin{align*}
			\|T_{z,\varepsilon} F\|_{L^2(\mathbb{R})}&=\|\mathcal{W}_z(-t)T_{z,\varepsilon} F\|_{L^2(\mathbb{R})}\\
			&=\left\|\int_{|t|>\varepsilon} s^{-1+i\frac{Im(z)}{\nu}} G(\cdot,t-s)ds\right\|_{L^2(\mathbb{R})},
		\end{align*}
		where $G(x,t)=\mathcal{W}_{2z}(-t)F(x,t)$ and
		\begin{equation}\label{norm-1}
			\|G(\cdot,t)\|_{L^2(\mathbb{R})}=\|F(\cdot,t)\|_{L^2(\mathbb{R})}.  
		\end{equation} If we define
		\begin{equation*}
			H_{z,\varepsilon}: h(t)\mapsto \int_{|t|>\varepsilon} s^{-1+iIm(z)} h(t-s)ds,
		\end{equation*}
		then we have
		\begin{equation}\label{norm1-1}
			\|T_{z,\varepsilon} F\|_{L^2(\mathbb{R},L^2(\mathbb{R}))}=\|H_{z,\varepsilon} G\|_{L^2(\mathbb{R},L^2(\mathbb{R}))}.
		\end{equation}
		Since the operator $H_{z,\varepsilon}$ is just Hilbert transform up to $iIm(z)$,  
		the operator $H_{z,\varepsilon}: L^2(\mathbb{R})\mapsto L^2(\mathbb{R})$    is bounded  with some constant depending only on $Im(z)$ exponentially. For further details, see Vega \cite{Vega}. Therefore, combining \eqref{norm-1} and \eqref{norm1-1}, we obtain
		\begin{align*}
			\|T_{z,\varepsilon} F\|_{L^2(\mathbb{R},L^2(\mathbb{R}))}&\leq C(Im(z))\|G\|_{L^2(\mathbb{R},L^2(\mathbb{R}))}\\
			&=C(Im(z))\|F\|_{L^2(\mathbb{R},L^2(\mathbb{R}))},
		\end{align*}
		which is our desired bound for $T_{z,\varepsilon}: L^2(\mathbb{R}, L^2(\mathbb{R}))\rightarrow L^2(\mathbb{R}, L^2(\mathbb{R}))$ when $Re(z)=0$.
	\end{proof}
	\begin{remark}
		One can see that, in Theorem \ref{One dim}, Kenig-Ponce-Vega obtained Strichartz type estimate for the operator $\mathcal{W}_\nu$ for dimension one involving a single function for any $\nu\in [0,1]$. However, its orthonormal version, i.e., Theorem \ref{real line}, holds only  for $ \nu\in [\frac{2}{3},1)$. This is due to the application of Hardy-Littlewood-Sobolev inequality.
	\end{remark}
	From \eqref{relation}, we also obtain the following orthonormal restriction theorem.
	\begin{theorem}\label{orthonormal-restriction}
		Suppose $\phi\in\mathcal{A}$, $\nu\in [0,1]$ and $(p,q)=\left(\frac{2}{1-\nu},\frac{4}{\nu}\right)$.
		Let $S=\{(\xi, \phi(\xi): \xi\in\mathbb{R}\}$
		be the noncompact hypersurface of $\mathbb{R}^2$ endowed with the measure $d\sigma_\nu(\xi,\phi(\xi))=\frac{|\phi''(\xi)|^\frac{\nu}{2}}{\sqrt{1+|\phi'(\xi)|^2}}d\xi$. 
		For any (possibly infinite) orthonormal system $\{F_j\}_{j\in J}$ in $L^2(S, d\sigma_\nu)$ and any sequence $\{n_j\}_{j\in J}$ in $\mathbb{C}$, we have
		\begin{equation*}
			\left\|\sum_{j\in J}n_j|E^\nu_SF_j|^2\right\|_{L^q(\mathcal{\mathbb{R}},L^p(\mathbb{R}))} \leq C\left\|\{n_j\}_{j\in J}\right\|_{\ell^\frac{2p}{p+1}},
		\end{equation*}
		where $C>0$ is independent of $\{F_j\}_{j\in J}$ and $\{n_j\}_{j\in J}$.
	\end{theorem}

	\subsection{Orthonormal Strichartz inequalities on $\mathbb{R}^d$, $d\geq2$}
	This subsection is devoted to extending the obtained orthonormal inequalities for $\mathcal{W}_\nu$ on the real line to higher dimensions.  The first problem in the higher dimension was to find the substitute for the second derivative that appeared in the dimension case; however, as explained in \cite{Kenig-Ponce-Vega}, the stationary phase lemma suggests that this substitute should be the Hessian of the phase function. 
	
	Let $P$ be a real polynomial on $\mathbb{R}^d$ and denote $HP=\det(\partial_{ij}^2P)_{i,j}$ is the determinant of
	the Hessian matrix of $P$. For $\nu\geq 0$ and $(x,t)\in\mathbb{R}^d\times\mathbb{R}$, we define
	\begin{equation*}
		\mathcal{W}_\nu(t) f(x)=\int_{\mathbb{R}^d} e^{i\left(t P(\xi)+x\cdot\xi\right)}|HP(\xi)|^\frac{\nu}{2} \hat{f}(\xi) \psi_P(\xi)d\xi,
	\end{equation*}
	where $\psi_P$ is a suitable cut-off function.  If $P$ is a real polynomial, it is not easy to see what the Hessian of $P$ looks like, but,  things will be simple if we consider    $P$  to be a a real elliptic polynomial.  	When $P$ is a real elliptic polynomial of degree $m$ on $\mathbb{R}^d$, i.e., the principle part of $P$ never vanishes outside of the origin, then $\psi_P\in C^\infty(\mathbb{R}^d)$ satisfies that 
	$$\psi_P(\xi)=\begin{cases}1\quad \text{ if} \quad |P(\xi)|\geq C|\xi|^m,\\
		0\quad \text{ if} \quad  |P(\xi)|\leq \frac{C}{2}|\xi|^m.
	\end{cases}$$
	When $P(\xi)=\sum\limits_{j=1}^d P_j(\xi_j)$ and $P_j$ is an arbitrary real polynomial of one variable with degree $P_j\geq2$ for all $1\leq j\leq d$, then $\psi_P\equiv1$.
	
	The following smoothing effect that involves a single function has been proved by  Kenig-Ponce-Vega \cite{Kenig-Ponce-Vega}.
	\begin{theorem}\label{Smoothing}
		Let $P$ be a real elliptic polynomial of $m\geq2$ or $P(\xi)=\sum\limits_{j=1}^d P_j(\xi_j)$, where $P_j$ is an arbitrary real polynomial of one variable with degree $P_j\geq2$ for all $1\leq j\leq d$. Then for any $0\leq \nu<\frac{2}{d}$ and $(p,q)=\left(\frac{2}{1-\nu},\frac{4}{d\nu}\right)$,
		\begin{align*}
			\|\mathcal{W}_\frac{\nu}{2}(t) f\|_{L^q(\mathbb{R},L^p(\mathbb{R}^d))}&\leq C\|f\|_{L^2(\mathbb{R}^d)},\\
			\left\|\int_{\mathbb{R}}\mathcal{W}_\nu(t-s)F(\cdot, s)ds\right\|_{L^q(\mathbb{R},L^p(\mathbb{R}^d))}&\leq C\|F\|_{L^{q'}(\mathbb{R},L^{p'}(\mathbb{R}^d))},
		\end{align*}
		for a constant $C>0$ independent of $f$ and $F$.
	\end{theorem}
	\begin{remark}
		It is pointed out by Keel-Tao \cite{KT} that for $d\geq 3$, the condition can be loosened to $0\leq \nu\leq\frac{2}{d}$.
	\end{remark}
	Our main aim is to extend Theorem \ref{Smoothing} for a system of orthonormal functions. The following decay of a class of oscillatory integrals is one of the key ingredients to proving the orthonormal inequality. 
	\begin{lemma} \label{decay-d-1}
		Let $\Omega$ be an open ball or complement of a ball in $\mathbb{R}^d$. Consider $\phi$, a real $C^{d+1}$-function defined on $\Omega$ such that for some $m\geq2$, there are constants $C_1,C_2>0$ such that 
		\begin{align*}
			C_1|\xi|^{m-1}&\leq \nabla \phi(\xi)|\leq C_2|\xi|^{m-1},\\
			C_1|\xi|^{(m-1)d}&\leq |H\phi(\xi)|=|\det(\partial_{ij}^2\phi)_{i,j}(\xi)|\leq C_2|\xi|^{(m-1)d},\\
			|\partial^\alpha \phi(\xi)|&\leq C_2 |\xi|^{m-\alpha} \text{ for all $\alpha$ with } |\alpha|\leq d+1.
		\end{align*}
		Then for $Re(z)=1$
		\begin{equation*}
			\left|\int_\Omega e^{i(t\phi(\xi)+x\cdot\xi)}\psi(\xi)|H\phi(\xi)|^{\frac{z}{2}}d\xi\right|\leq C(1+|Im(z)|)^d|t|^{-\frac{d}{2}},
		\end{equation*}
		where $(x,t)\in \mathbb{R}^d\times\mathbb{R}$, $\psi$ is a smooth function vanishing with all its derivatives on the boundary of $\Omega$ and $C$ is a constant which does not depend on $x,t$ or $Im(z)$.
	\end{lemma}
	
	\begin{lemma}\label{decay-d-2}
		If $P(\xi)=\sum\limits_{j=1}^d P_j(\xi_j)$, where $P_j$ is an arbitrary real polynomial of one variable with degree $P_j\geq2$ for all $1\leq j\leq d$, then for $Re(z)=1$,
		\begin{equation*}
			\left|\int_{\mathbb{R}^d} e^{i\left(t P(\xi)+x\cdot\xi\right)}|HP(\xi)|^\frac{z}{2}  d\xi\right|\leq C(1+|Im(z)|)^d|t|^{-\frac{d}{2}},
		\end{equation*}
		where $C$ is a constant which does not depend on $x,t$ or $Im(z)$.
	\end{lemma}
	\begin{proof}
		Note that $HP(\xi)=\prod\limits_{j=1}^d P''_j(\xi_j)$ and we can write
		\begin{equation*}
			\int_{\mathbb{R}^d} e^{i\left(t P(\xi)+x\cdot\xi\right)}|HP(\xi)|^\frac{z}{2}  d\xi=\prod\limits_{j=1}^d \int_{\mathbb{R}} e^{i(tP_j(\xi_j)+x_j\xi_j)}|P''_j(\xi_j)|^\frac{z}{2}d\xi_j. 
		\end{equation*}
		Observe that each $P_j\in\mathcal{A}$ for all $1\leq j\leq d$ and the desired result follows immediately from Lemma \ref{decay-1}.
	\end{proof}
	Finally, argued similarly to Theorem \ref{real line}, we obtain the following orthonormal Strichartz inequalities for  $\mathcal{W}_\frac{\nu}{2}$ in higher dimensions.
	\begin{theorem}\label{Higher dim}
		Let $P$ be a real elliptic polynomial of $m\geq2$ or $P(\xi)=\sum\limits_{j=1}^d P_j(\xi_j)$, where $P_j$ is an arbitrary real polynomial of one variable with degree $P_j\geq2$ for all $1\leq j\leq d$. Assume $\nu\in [\frac{2}{d+2},\frac{2}{d+1})$ and $(p,q)=\left(\frac{1}{1-\nu},\frac{2}{d\nu}\right)$. For any (possibly infinite) orthonormal system $\{f_j\}_{j\in J}$ in $L^2(\mathbb{R}^d)$ and any sequence $\{n_j\}_{j\in J}$ in $\mathbb{C}$, we have
		\begin{equation*}
			\left\|\sum_{j\in J}n_j|\mathcal{W}_\frac{\nu}{2}(t)f_j|^2\right\|_{L^q(\mathcal{\mathbb{R}},L^p(\mathbb{R}^d))} \leq C\left\|\{n_j\}_{j\in J}\right\|_{\ell^\frac{2p}{p+1}},
		\end{equation*}
		where $C>0$ is independent of $\{f_j\}_{j\in J}$ and $\{n_j\}_{j\in J}$.
	\end{theorem}
	\begin{proof} 
		Similar to the  one dimensional  real line $\mathbb{R}$, for $Re(z)\geq 0$, we also define
		\begin{equation*}
			I_z(x,t)=\int_{\mathbb{R}^d} e^{i\left(t P(\xi)+x\cdot\xi\right)}|HP(\xi)|^\frac{z}{2}  \psi_P(\xi)d\xi, \quad    (x,t)\in\mathbb{R}^d\times\mathbb{R}.
		\end{equation*} 
		Then  from Lemma \ref{decay-d-1} and \ref{decay-d-2}, it follows that
		\begin{equation}\label{Decay-2}
			|I_z(x,t)|\leq C(1+|Im(z)|)^d|t|^{-\frac{d}{2}}.
		\end{equation}
		for $Re(z)=1$. Therefore, using the decay estimate \eqref{Decay-2}, the desired result automatically follows in a similar way that we used in Theorem \ref{real line}. and we omit the further details.
	\end{proof}
	\begin{remark}
		Notice that  Strichartz type estimate involving a single function for the operator $\mathcal{W}_\nu$ was proved by  Kenig-Ponce-Vega in Theorem \ref{Smoothing} for any $0\leq \nu<\frac{2}{d}$. However, due to the application of Hardy-Littlewood-Sobolev inequality, its orthonormal version (see  Theorem \ref{Higher dim}) holds only  for  $\nu\in [\frac{2}{d+2},\frac{2}{d+1})$. 
	\end{remark}
	\begin{remark}
		Similar to  Theorem \ref{orthonormal-restriction}, Theorem \ref{Higher dim} can also be seen as an orthonormal restriction theorem for the Fourier extension operator on noncompact hypersurfaces of the form $S=\{(\xi, P(\xi)) : \xi \in \mathbb{R}^{d}  \}$, where  $P$ be a real elliptic polynomial of $m\geq2$ or $P(\xi)=\sum\limits_{j=1}^d P_j(\xi_j)$ and $P_j$'s  are arbitrary real polynomial of one variable with degree $P_j\geq2$ for all $1\leq j\leq d$.
	\end{remark}
	
	\section{Appendix}  \label{sec8}
	In this section, we prove  Theorem \ref{Strichartz-single},  i.e.,  Strichartz estimates for a single function under the Assumption   $(\mathcal{A})$ and $(\mathcal{B})$. From  \cite{GV},  first, we recall some elementary abstract results.  For any vector space $\mathcal{D}$, we denote  $\mathcal{D}_a^*$  as the  algebraic dual of $\mathcal{D}$,  $\mathcal{L}_a(\mathcal{D}, X)$ as the space of linear maps from $\mathcal{D}$ to some other vector space $X$, and finally  $\langle\varphi, f\rangle_{\mathcal{D}}$ denotes the pairing between $\mathcal{D}_a^*$ and $\mathcal{D}$ with $f \in \mathcal{D}$ and $\varphi \in \mathcal{D}_a^*$, which we take to be linear in $f$ and anti linear in $\varphi$.    Then under these notations, the following lemma (see \cite{GV}) will be useful to prove  Theorem \ref{Strichartz-single}.
	\begin{lemma}\label{equ} Let $\mathcal{H}$ be a Hilbert space, $\mathcal{X}$ be a Banach space, $\mathcal{X}^*$ is the dual of $\mathcal{X}$, and $\mathcal{D}$ be a vector space densely contained in $\mathcal{X}$. Let $T\in \mathcal{L}_a(\mathcal{D},\mathcal{H})$ and   $T^*\in \mathcal{L}_a(H,\mathcal{D}_a^*)$ be its adjoint operator, defined by
		\begin{equation*}
			\langle T^*v,f \rangle_\mathcal{D}=\langle v,Tf \rangle_\mathcal{H}, \quad \forall f\in \mathcal{D}, \quad \forall v \in \mathcal{H},
		\end{equation*}
		where  $\langle\cdot,\cdot\rangle_\mathcal{H}$ is the scalar product in $\mathcal{H}$. Then the following three conditions are equivalent:
		
		(1) There exists  $0\leqslant a \leqslant \infty$ such that for all $f \in \mathcal{D}$,
		\begin{equation*}
			\|Tf\| \leqslant a\|f\|_X.
		\end{equation*}
		
		(2) $\mathfrak{R}(T^*)\subset \mathcal{X}^*$ and there exists   $0\leqslant a \leqslant \infty$, such that for all $v \in \mathcal{H}$,
		\begin{equation*}
			\|T^*v\|_{\mathcal{X}^*} \leqslant a\|v\|.
		\end{equation*}
		
		(3) $\mathfrak{R}(T^*T)\subset \mathcal{X}^*$,and there exists $a$, $0\leqslant a \leqslant \infty$, such that for all $f \in \mathcal{D}$,
		\begin{equation*}
			\|T^*Tf\|_{\mathcal{X}^*} \leqslant a^2\|f\|_\mathcal{X},
		\end{equation*}
		where $\|\cdot\|$ denote the norm in $\mathcal{H}$. The constant $a$ is the same in all three parts. If one of those conditions is satisfied, the operators $T$ and $T^*T$ extend by continuity to bounded operators from $\mathcal{X}$ to $\mathcal{H}$ and from $\mathcal{X}$ to $\mathcal{X}^*$, respectively.
	\end{lemma}
	\begin{lemma} \label{DuilatyXX} Let $\mathcal{H}$, $\mathcal{D}$ and two triplets $(\mathcal{X}_i,T_i,a_i), i=1,2$, satisfy any of the conditions of Lemma \ref{equ}. Then for all choices of $i,j=1,2$, $\mathfrak{R}(T_i^*T_j) \subset \mathcal{X}_i^*$, and for all $f \in \mathcal{D}$, we have
		\begin{equation*}
			\|T_i^* T_jf\|_{\mathcal{X}_i^*} \leqslant a_i a_j\|f\|_{\mathcal{X}_j}.
		\end{equation*}
	\end{lemma}
	\begin{lemma}\label{L1LW}  Let $\mathcal{H}$ be a Hilbert space, let $I$ be an interval of $\mathbb{R}$, let $\mathcal{X} \subset \mathcal{S}'(I \times \mathbb{R}^n)$ be a Banach space, let $\mathcal{X}$ be stable under time restriction, let $\mathcal{X}$ and $T$ satisfy (any of) the conditions of Lemma \ref{equ}. Then the operator $T^*T$ is a bounded operator from $L_t^1(I,\mathcal{H})$ to $\mathcal{X}^*$ and from $\mathcal{X}$ to $L_t^\infty(I,\mathcal{H})$.
	\end{lemma}
	Now we are in a position to prove  Theorem \ref{Strichartz-single}, but we provide only the main important steps.  
	\begin{proof}[Proof of  Theorem \ref{Strichartz-single}]
		\textbf{Step 1:} It follows from $(\mathcal{A})$ that the Schr\"{o}dinger operator $e^{itL}$ extends to a bounded operator from $L^1(X)$ to $L^\infty(X)$ with
		\begin{equation*}
			\|e^{itL}f\|_{L^\infty(X)}\lesssim |t|^{-\frac{n}{2}}\|f\|_{L^1(X)}.
		\end{equation*}
		The Schr\"{o}dinger operator $e^{itL}$ is a unitary operator on $L^2(X)$ and we have
		\begin{equation*}
			\|e^{itL}f\|_{L^2(X)}= \|f\|_{L^2(X)}.
		\end{equation*}
		By Riesz-Thorin interpolation theorem, we deduce that the Schr\"{o}dinger operator $e^{itL}$ extends to a bounded operator from $L^{p'}(X)$ to $L^p(X)$ with
		\begin{equation}\label{Decay}
			\|e^{itL}f\|_{L^p(X)}\lesssim |t|^{-n\left(\frac{1}{2}-\frac{1}{p}\right)} \|f\|_{L^{p'}(X)},
		\end{equation}
		for any $p\geq 2$.
		
		\textbf{Step 2:} We prove the diagonal Strichartz estimate. For the particular case $p_1=p_2=p\geq 2$, $q_1=q_2=q$ and $0\leq\frac{2}{q}=n\left(\frac{1}{2}-\frac{1}{p}\right)<1$, by \eqref{Decay}, Minkowski's inequality and Hardy-Littlewood-Sobolev inequality, we obtain
		\begin{align*}
			\left\|\int_0^te^{i(t-s)L}F(\cdot,s)ds\right\|_{L^q\left(\mathcal{I},L^p(X)\right)}&\lesssim \left\|\int_0^t\frac{\|F(\cdot,s)\|_{L^{p'}(X)}}{|t-s|^{n\left(\frac{1}{2}-\frac{1}{p}\right)}}ds\right\|_{L^q(\mathcal{I})}\\
			&\lesssim \|F\|_{L^{q'}\left(\mathcal{I},L^{p'}(X)\right)}.
		\end{align*}
		
		\textbf{Step 3:} By the same procedure as Step 2, we also prove that 
		\begin{equation}\label{Factorization}
			\left\|\int_\mathcal{I}e^{i(t-s)L}F(\cdot,s)ds\right\|_{L^q\left(\mathcal{I},L^p(X)\right)}
			\lesssim \|F\|_{L^{q'}\left(\mathcal{I},L^{p'}(X)\right)}.
		\end{equation}
		
		\textbf{Step 4:} Let $\mathcal{H}=L^2(X)$ and $\mathcal{X}=L^{q'}\left(\mathcal{I},L^{p'}(X)\right)$ in Lemma \ref{equ}. We shall define the bounded operator
		\begin{align*}
			T: L^{q'}\left(\mathcal{I},L^{p'}(X)\right)&\rightarrow L^2(X)\\
			F(x,t)&\mapsto \int_\mathcal{I}  e^{-isL}F(x,s)ds,
		\end{align*}
		whose adjoint operator is
		\begin{align*}
			T^*: L^2(X)&\rightarrow L^q\left(\mathcal{I},L^p(X)\right)\\
			f(x)&\mapsto e^{itL}f(x),
		\end{align*}
		where duality is define by the scalar products in $L^2(X)$ and in $L^2\left(\mathcal{I},L^2(X)\right)$. By factorization of $T^*T$, we have
		\begin{equation*}
			T^*TF(x, t)=\int_\mathcal{I}e^{i(t-s)L}F(x,s)ds.
		\end{equation*}
		The homogeneous Strichartz estimate and the dual homogeneous Strichartz estimate come out immediately from the obtained result \eqref{Factorization} and Lemma \ref{equ}. Using Lemma \ref{DuilatyXX} and \eqref{Factorization}, we get
		\begin{equation}\label{non-retarded}
			\left\|\int_\mathcal{I}e^{i(t-s)L}F(\cdot,s)ds\right\|_{L^{q_2}\left(\mathcal{I},L^{p_2}(X)\right)}\lesssim \|F\|_{L^{q_2'}\left(\mathcal{I},L^{p_1'}(X)\right)},
		\end{equation}
		unfortunately, which is not the integral operator in the retarded Strichartz estimate because of the "retardation" in time. For any function $F$ on $\mathcal{I}\times X$, we can rewrite the integral operator by
		\begin{equation*}
			\int_0^te^{i(t-s)L}F(x,s)ds=\int_\mathcal{I} \chi_{\mathbb{R}^+}(t-s)e^{i(t-s)L}\tilde{F}(x,s)ds,\;\forall t\in \mathcal{I},
		\end{equation*}
		where $\chi_{\mathbb{R}^+}$ is the characteristic function on $\mathbb{R}^+$ and $\tilde{F}$ is the continuation of $F$ by zero over $\mathbb{R}^+\times X$.
		Due to the fact that the space $\mathcal{X}=L^{q'}\left(\mathcal{I}, L^{p'}(X)\right)$ is unstable under restriction in time (which means that the multiplication by the characteristic function of an interval in time is a bounded operator with norm uniformly bounded with respect to the interval), we can overcome this technical multiplication and prove the retarded Strichartz estimates by the same procedure as \eqref{non-retarded}.
	\end{proof}
	
	\section*{Acknowledgments} MS is supported by the Guangdong Basic and Applied Basic Research Foundation (No. 2023A1515010656). HW is supported by the National Natural Science Foundation of China (Grant No. 12171399 and 12271041). SSM is supported by the DST-INSPIRE Faculty Fellowship
	DST/INSPIRE/04/2023/002038.

\end{document}